\documentclass[12pt]{amsart}
\usepackage[nobysame,abbrev,alphabetic]{amsrefs}
\usepackage{amssymb}
\usepackage[all]{xy}
\usepackage{stmaryrd}
\usepackage{titlesec}
\usepackage{titletoc}

\titleformat{\section}
  {\normalfont\large\bfseries} 
  {\thesection}                
  {1em}                        
  {}                           
\titleformat{\subsection}
  {\normalfont\large} 
  {\thesubsection}    
  {1em}               
  {}                  
  
\dottedcontents{section}[2.5em]{}{2.3em}{1pc}

\DeclareMathOperator{\Aut}{Aut}

\DeclareMathOperator{\Char}{Char}

\DeclareMathOperator{\Coker}{Coker}

\DeclareMathOperator{\ET}{ET}

\DeclareMathOperator{\Gal}{Gal}
\DeclareMathOperator{\Hom}{Hom}

\DeclareMathOperator{\Img}{Im}

\DeclareMathOperator{\Ker}{Ker}

\DeclareMathOperator{\rank}{rank}
\DeclareMathOperator{\res}{res}
\DeclareMathOperator{\Res}{Res}

\DeclareMathOperator{\trdeg}{tr.deg}
\DeclareMathOperator{\Tr}{Tr}

\newcommand{\nek}{,\ldots,}

\newcommand{\inv}{^{-1}}
\newcommand{\isom}{\cong}

\newcommand{\tensor}{\otimes}

\newtheorem{thm}{Theorem}[section]
\newtheorem{cor}[thm]{Corollary}
\newtheorem{lem}[thm]{Lemma}
\newtheorem{prop}[thm]{Proposition}
\newtheorem{defin}[thm]{Definition}
\newtheorem{exam}[thm]{Example}
\newtheorem{examples}[thm]{Examples}
\newtheorem{rem}[thm]{Remark}
\newtheorem{case}{\sl Case}

\newtheorem{conj}[thm]{Conjecture}

\swapnumbers
\newtheorem{rems}[thm]{Remarks}

\numberwithin{equation}{section}

\newcommand{\alp}{\alpha}

\newcommand{\Gam}{\Gamma}

\newcommand{\eps}{\epsilon}
\newcommand{\lam}{\lambda}

\newcommand{\sig}{\sigma}

\newcommand{\dbC}{\mathbb{C}}
\newcommand{\dbF}{\mathbb{F}}

\newcommand{\dbQ}{\mathbb{Q}}
\newcommand{\dbR}{\mathbb{R}}

\newcommand{\dbZ}{\mathbb{Z}}

\newcommand{\grm}{\mathfrak{m}}

\newcommand{\calE}{\mathcal{E}}
\newcommand{\calG}{\mathcal{G}}

\newcommand{\calZ}{\mathcal{Z}}

\begin{document}

\title[The Elementary Type Conjecture]{The Elementary Type Conjecture for Maximal Pro-$p$ Galois groups}

\author{Ido Efrat}
\address{Earl Katz Family Chair in Pure Mathematics\\
Department of Mathematics\\
Ben-Gurion University of the Negev\\
P.O.\ Box 653, Be'er-Sheva 84105\\
Israel} \email{efrat@bgu.ac.il}

\thanks{This work was supported by the Israel Science Foundation (grant No.\ 569/21). }


\maketitle
\begin{center}
\itshape
Dedicated to J\'an Min\'a\v c on his 71th birthday
\end{center}

\begin{abstract}
The Elementary Type Conjecture in Galois theory provides a concrete inductive description of the finitely generated maximal pro-$p$ Galois groups $G_F(p)$ of fields $F$ containing a root of unity of order $p$.
We describe several variants of this conjecture, and prove various connections between these variants and other conjectures in field and Galois theory.
We focus on a strong arithmetical variant of the conjecture, and its implications to the realization of pro-$p$ Demu\v skin groups as Galois groups.
\end{abstract}


\tableofcontents

\section{Introduction}
\subsection{The general philosophy}
A fortunate, yet rare, situation in mathematics is where all objects in a certain category can be constructed in finitely many steps from a well-understood list of building blocks using a well-understood list of operations.
For instance, finite-dimensional linear spaces over a field $F$ are isomorphic to direct sums of finitely many copies of $F$.
Similarly, finitely generated abelian groups are the direct sums of cyclic groups of the forms $\dbZ$ or $\dbZ/q$, where $q$ is a prime power.
The situation differs somewhat for finite groups: while their building blocks -- the finite simple groups -- have been classified, constructing arbitrary finite groups from these `atoms' via composition series can be highly complex.

Fields, by contrast, possess a profoundly rich structure, rendering a classification of this kind seemingly unattainable. 
Similarly, canonical Galois groups of fields, such as their absolute Galois groups, are not expected to exhibit such a neat organization. 
In fact, characterizing the profinite groups realizable as absolute Galois groups remains a wide-open problem.

However, in this paper we describe a similar \textsl{conjectural} phenomenon in the context of fields, where certain canonical Galois groups, if finitely generated, can  expectedly be constructed from a well-understood list of `atoms' using two simple group-theoretic operations. 

More specifically, we fix throughout the paper a prime number $p$ and consider fields $F$ containing a root of unity of order $p$.
Let $G_F(p)=\Gal(F(p)/F)$ be the maximal pro-$p$ Galois group of $F$, where $F(p)$ stands for the maximal pro-$p$ Galois extension of $F$.
Thus $G_F(p)$ is the maximal pro-$p$ quotient of the absolute Galois group $G_F=\Gal(F^{\rm sep}/F)$ of $F$, where $F^{\rm sep}$ denotes its separable closure.
The group-theoretic structure of $G_F(p)$ as a pro-$p$ group is also a mystery in general.
But if we restrict ourselves to such groups which are \textsl{finitely generated} (as pro-$p$ groups), then the \textsl{Elementary Type Conjecture} for pro-$p$ Galois groups, described in this paper, predicts a complete classification of this subfamily in the above style.
See \S1.3 below for a formulation of this conjecture in its strongest (arithmetical) form. 

It is known that the group $G_F(p)$ encodes key arithmetic data on the arithmetical structure of $F$, such as orderings, Witt rings (when $p=2$), and certain valuations.
Thus this conjectural classification can be thought of as a partial classification of the arithmetic structure of fields with $(F^\times:(F^\times)^p)<\infty$ (by Kummer theory, this is equivalent to $G_F(p)$ being finitely generated).   
This conjectural description therefore serves as a very convenient `testing ground' for other conjectures in field theory and Galois cohomology.
See \S\ref{section on applications}.

This paper has multiple objectives:
First, we describe the conjecture,  its several variants, some of the existing evidence for its validity, and several applications to other open problems in the arithmetic of fields.
Second, we give detailed proofs of some related facts and connections, which seem to be known to experts, but which have no references in the relevant literature.
Further, we analyze connections between the Elementary Type Conjecture and other major problems in Galois theory and valuation theory. 
These include realization of Demu\v skin groups as Galois groups (see below),  construction of valuations from Galois-cohomological and Milnor $K$-theory data,  and connections with rigidity in fields.

\subsection{Brief history}
The Elementary Type Conjecture emerged from studying quadratic forms over fields of characteristic $\neq 2$.
In the 1970s and 1980s, equivalent formalisms -- quadratic form schemes, quaternionic maps, and abstract Witt rings -- sought to axiomatize this theory (see \cite{Marshall04}, \cite{Kula91}, \cite{Lam05}*{Ch.\ XII, \S8} and references therein). 
Within these frameworks, `local type' structures mimic quadratic forms over local and finite fields. 
Formal operations, called \textit{products} and \textit{extensions}, build the class of \textit{elementary type} structures from these local types. 
Such structures are always realizable over fields $F$ of characteristic $\neq2$ with $(F^\times:(F^\times)^2)<\infty$, raising the question: 
Are all such structures which are realizable over fields of elementary type? 
(See \cite{Marshall04}*{\S3.4} and \cite{Lam05}*{p.\ 463}.) 
This was confirmed using elementary tools in many cases, like $(F^\times:(F^\times)^2)\leq 128$ (see \S\ref{section on evidence}).

Another path to the conjecture arose from \textit{Pythagorean fields} -- fields $F$ of characteristic $\neq2$ satisfying $F^2=F^2+F^2$, or equivalently, such that $G_F(2)$ is generated by elements of order $2$.
An analysis of the ordered field structure of such fields due to Br\"ocker (\cite{Broecker76}, \cite{Broecker77}), Craven \cite{Craven78}, and Kula \cite{Kula79} revealed that, under the finiteness assumption, they have an elementary type in the appropriate sense.
Moreover, they showed that in the Pythagorean case the operation of extension can be interpreted valuation-theoretically.
This was further studied by Jacob \cite{Jacob81} and Min\'a\v c \cite{Minac86}, who interpreted the elementary type decomposition in the Pythagorean case in terms of the maximal pro-$2$ Galois group $G_F(2)$.
See also \cite{Efrat93} and \cite{EfratHaran94} for generalization to the non-finitely generated Pythagorean case.

The next major advance was due to Jacob and Ware (\cite{JacobWare89}, \cite{JacobWare91}), who translated the conjectural elementary type decomposition of Witt rings of arbitrary fields $F$ of characteristic $\neq2$ and $(F^\times:(F^\times)^2)<\infty$ into a decomposition theory of finitely generated maximal pro-$2$ Galois groups with their cyclotomic pro-$2$ character (see \S\ref{section on Galois cyclotomic pro p pairs}).

\subsection{The arithmetical approach}
In \cite{Efrat95} we suggested the following stronger \textsl{arithmetical version} of the Elementary Type Conjecture, for an arbitrary prime number $p$:

\begin{conj}
\label{ETC introduction}
Let $F$ be a field containing a root of unity of order $p$ and such that $G_F(p)$ is finitely generated (equivalently, $(F^\times:(F^\times)^p)<\infty$).
Then $G_F(p)$ is the free product in the category of pro-$p$ groups of finitely many (closed) subgroups of the following forms:
\begin{enumerate}
\item
$\dbZ_p$;
\item
Decomposition groups of valuations on $F$ with a nontrivial inertia group in $F(p)$;
\item
When $p=2$, copies of $\dbZ/2$ (equivalently, Galois groups of relative real closures  in $F(2)$ of orderings on $F$). 
\end{enumerate}
\end{conj}
  
As explained in \S\ref{section on the elementary type conjecture}, this conjecture implies a `Lego form' description of $G_F(p)$:
It can be inductively constructed in finitely many steps from the maximal pro-$p$ Galois groups of $\dbC$, $\dbR$ (when $p=2$), the finite extension of $\dbQ_p$ containing a $p$th root of unity, and finite fields, using free pro-$p$ products and certain semi-direct products $\dbZ_p\rtimes\cdot$.
Again, to make this inductive description precise one also needs to keep track of the cyclotomic action.

This arithmetical viewpoint enabled progress on Conjecture \ref{ETC introduction} using valuation-theoretic machinery.
E.g., it was proved for algebraic extensions of fields for which the Hasse principle for the Brauer groups holds for every finite extension.
These include global fields, function fields in one variable over local fields, and function fields in one variable over pseudo algebraically closed (PAC) fields (see \S\ref{section on evidence}).

\subsection{Demu\v skin groups as Galois groups}  
Conjecture \ref{ETC introduction} is closely related to another important open problem in infinite Galois theory, namely, the appearance of pro-$p$ Demu\v skin groups as maximal pro-$p$ Galois groups of fields.
These groups were investigated in classical works by Serre \cite{Serre63} and Labute \cite{Labute67}, and we refer to \S\ref{section on Demushkin groups} for their precise cohomological definition.
When $F$ is a finite extension of $\dbQ_p$ containing a $p$th root of unity, $G_F(p)$ is Demu\v skin;
however not all Demu\v skin groups are realizable in this manner.
It is currently unknown which pro-$p$ Demu\v skin groups are realizable as $G_F(p)$ for some $F$, and in this case, what is the arithmetical structure of $F$.
As we shall see in \S\ref{section on arithmetically Demushkin fields}, Conjecture \ref{ETC introduction} completely answers these questions.
Specifically, it implies that if $G_F(p)$ is Demu\v skin, then $F$ is an \textsl{arithmetically Demu\v skin field}, which, loosely speaking, means that it is endowed with a valuation which resembles a $p$-adic valuation (see \S\ref{section on arithmetically Demushkin fields} for the precise definition).
Moreover, in  \S\ref{section on equivalence of conjectures} we prove that Conjecture \ref{ETC introduction} is in fact \textsl{equivalent} to the weaker `Lego form' of the conjecture, as described above, plus the claim that if $G_F(p)$ is Demu\v skin, then $F$ is arithmetically Demu\v skin. 

I thank the referee for their comments.

\section{Cyclotomic Pro-$p$ Pairs}
\label{section on cyclotomic pro-p pairs}
We write $\dbZ_p^\times$ for the group of units of $\dbZ_p$ and $\dbZ_p^{\times,1}=1+p\dbZ_p$ for its subgroup of $1$-units.  

For reasons that will become clear below, we will work here not in the category of pro-$p$ groups, but rather in a category of pairs $\calG=(G,\theta)$, where $G$ is a pro-$p$ group and  $\theta\colon G\to \dbZ_p^{\times,1}$ is a continuous homomorphism.
We call $\calG$ a \textsl{cyclotomic pro-$p$ pair} (some authors use the term \textsl{an oriented pro-$p$ group}).
We refer to $\theta$ as the \textsl{character} of $\calG$.
A \textsl{morphism} $\varphi\colon (G_1,\theta_1)\to(G_2,\theta_2)$ of cyclotomic pro-$p$ pairs is a continuous homomorphism $\varphi\colon G_1\to G_2$ such that $\theta_1=\theta_2\circ\varphi$.

We list some `atomic' cyclotomic pro-$p$ pairs to be used in the sequel:

\begin{enumerate}
\item[(1)]
The \textsl{trivial pair} is $(1,1)$, i.e., $G$ is the trivial group and $\theta$ is the trivial homomorphism.
\item[(2)]
For every $\alp\in\dbZ_p^{\times,1}$ we have the pair $\calZ^\alp=(\dbZ_p,\theta)$, where $\theta\colon\dbZ_p\to\dbZ_p^{\times,1}$ is the continuous homomorphism mapping the generator $1$ of $\dbZ_p$ to $\alp$.
\item[(3)]
When $p=2$ we have the pair
\[
\calE=(\dbZ/2,\theta),
\]
where $\theta\colon \dbZ/2\to\{\pm1\}\leq\dbZ_2^\times$ is the nontrivial homomorphism.
\end{enumerate}
An additional family of `atomic' pairs, namely the cyclotomic pro-$p$ pairs of \textsl{$p$-adic type}, will be discussed in \S\ref{section on Demushkin groups}.

In the category of cyclotomic pro-$p$ pairs we have two fundamental constructions:

The \emph{free product} of cyclotomic pro-$p$ pairs $\calG_1=(G_1,\theta_1)$ and $\calG_2=(G_2,\theta_2)$ is the pro-$p$ pair
\[
\calG_1*\calG_2=(G_1*_pG_2,\theta_1*_p\theta_2),
\]
where $G_1*_pG_2$ is the free product in the category of pro-$p$ groups, and $\theta_1*_p\theta_2\colon G_1*_pG_2\to \dbZ_p^{\times,1}$ is the unique homomorphism of pro-$p$ groups that extends $\theta_1$ and $\theta_2$, as provided by the universal property of $G_1*_pG_2$.
This construction has the usual universal property of coproducts in the category of cyclotomic pro-$p$ pairs.

Next let $A$ be a free abelian pro-$p$ group.
Thus $A\cong\dbZ_p^m$ for some cardinal number $m$, and thus we have a multiplication map $\dbZ_p^\times\times A\to A$.
Given a cyclotomic pro-$p$ pair $\bar\calG=(\bar G,\bar\theta)$, the \emph{extension} of $\bar\calG$ by $A$ is the pair
\[
A\rtimes\bar\calG=(A\rtimes \bar G,\theta),
\]
where $A\rtimes \bar G$ is the semi-direct product with respect to the action
\[
\bar ga\bar g^{-1}=\bar\theta(\bar g)a
\]
for $\bar g\in\bar G$ and $a\in A$, and where $\theta\colon A\rtimes \bar G\to \dbZ_p^{\times,1}$ is the composition of the projection $\pi\colon A\rtimes \bar G\to\bar G$ with $\bar\theta\colon\bar G\to \dbZ_p^{\times,1}$.
Then $\pi\colon\calG\to\bar\calG$ is an epimorphism of pro-$p$ pairs.

If $A'$ is another free abelian pro-$p$ group, then
\begin{equation}
\label{associativity of extension}
(A\times A')\rtimes\bar\calG\isom A\rtimes(A'\rtimes\bar\calG).
\end{equation}

\begin{exam}
\label{Z2 E is E E}
\rm
When $p=2$, $\dbZ_2\rtimes\calE\isom\calE*\calE$. 
\end{exam}

To any cyclotomic pro-$p$ pair $\calG=(G,\theta)$ we associate a homomorphism 
\begin{equation}
\label{eps calG}
\eps_\calG\colon G\to\dbZ/p
\end{equation}
as follows:
If $p=2$, it is the composition of $\theta$ with the projection $\dbZ_2^{\times,1}=\{\pm1\}\times(1+4\dbZ_2)\to\{\pm1\}$, and if $p>2$ we simply take $\eps_\calG=0$.
Note that when $p=2$ and $\Img(\theta)\not\subseteq1+4\dbZ_2$, the map $\eps_\calG$ is a surjective morphism $\calG\to\calE$.

Given a cyclotomic pro-$p$ pair $\calG=(G,\theta)$, we write $\dbZ_p(1)$ for the $G$-module with underlying group $\dbZ_p$ and $G$-action defined by $g\cdot a=\theta(g)a$ for $g\in G$ and $a\in \dbZ_p$.
For $k\geq1$ we similarly define a $G$-module $\dbZ/p^k(1)$.
Thus $\dbZ/p(1)=\dbZ/p$ with the trivial $G$-action.
We say that $\calG$ has the \textsl{formal Hilbert 90} property if for every $k\geq1$ the natural map of profinite cohomology groups $H^1(G,\dbZ/p^k(1))\to H^1(G,\dbZ/p)$ is surjective (note that this terminology slightly differs from that of \cite{MerkurjevScavia23}, where the surjectivity is assumed for every closed subgroup of $G$).

\section{Galois cyclotomic pro-$p$ pairs} 
\label{section on Galois cyclotomic pro p pairs} 
The motivating example of cyclotomic pro-$p$ pairs comes from Galois theory, as follows:

Let $F$ be a field containing a root of unity of order $p$.
In particular, $\Char\,F\neq p$.
We write $\mu_{p^r}$ for the group of all roots of unity of order $p^r$ in the separable closure $F^{\rm sep}$ of $F$, and set $\mu_{p^\infty}=\bigcup\mu_{p^r}$.
As in the Introduction, we write
\[
G_F(p)=\Gal(F(p)/F), \quad G_F=\Gal(F^{\rm sep}/F),
\]
for the maximal pro-$p$ Galois group and the absolute Galois group of $F$, respectively.
The \textsl{pro-$p$ cyclotomic character} of $G_F(p)$ is the composition
\[
\chi_F=\chi_{F,p}\colon G_F(p)\xrightarrow{\res}\Aut(\mu_{p^\infty}/\mu_p)\xrightarrow{\sim}\dbZ_p^{\times,1},
\]
where the latter isomorphism associates $\alp\in\dbZ_p^{\times,1}$ with the automorphism $\zeta\mapsto\zeta^\alp$ in $\Aut(\mu_{p^\infty}/\mu_p)$.

\begin{defin}
\label{pairs of Galois type}
\rm
The \textsl{Galois cyclotomic pro-$p$ pair} of $F$ is 
\[
\calG_F(p)=(G_F(p),\chi_F).
\]
We call such pairs \textsl{cyclotomic pro-$p$ pairs of Galois type}.
\end{defin}

\begin{rems}
\label{aaa}
\rm
(1) \quad
For every $k\geq1$ we have $\dbZ/p^k(1)=\mu_{p^k}$ as $G_F(p)$-modules. 
Recall the Kummer isomorphism
\[
F^\times/(F^\times)^{p^k}\xrightarrow{\sim}H^1(G_F(p),\mu_{p^k}), \quad a(F^\times)^{p^k}\mapsto (a)_F,
\]
where $(a)_F=(a)_{F,p^k}$ is the \textsl{Kummer element} of $a$.
Since the natural map $F^\times/(F^\times)^{p^k}\to F^\times/(F^\times)^p$ is surjective, so is the map $H^1(G_F(p),\mu_{p^k})\to H^1(G_F(p),\mu_p)$.
Further,  $H^1(G_F(p),\mu_p)=H^1(G_F(p),\dbZ/p)$ since $\mu_p\subseteq F$.
Thus $\calG_F(p)$ has the formal Hilbert 90 property.  

\medskip

(2) \quad
The homomorphism $\eps_\calG\colon G_F(p)\to\dbZ/p$  of (\ref{eps calG}) may be identified with the Kummer element $(-1)_F$ in $H^1(G_F(p),\dbZ/p)=\Hom(G_F(p),\dbZ/p)$.
\end{rems}

\begin{examples}
\label{examples of realizable pairs}
\rm  
(1) \quad
The cyclotomic pro-$p$ pair $\calG_\dbC(p)$ is the trivial pair.

\medskip

(2) \quad
For $p=2$ one has $\calG_\dbR(2)=\calE$.

\medskip

(3) \quad
Let $F=\dbF_{l^r}$ be a finite field containing a root of unity of order $p$, and where $l$ is a prime number and $r\geq1$.
Then $l^r\equiv1\pmod p$ and we have $\calG_F(p)\isom\calZ^{l^r}$.

\medskip

(4) \quad
More generally, for every $1\neq\alp\in\dbZ_p^{\times,1}$ there exist (many) fields $F$ such that the absolute Galois group $G_F$ is pro-$p$ (so in particular, $F$ contains a root of unity of order $p$) and $\calG_F(p)\isom\calZ^\alp$.
For instance, define $\sig\in \Gal(\dbQ(\mu_{p^\infty})/\dbQ(\mu_p))$ by $\sig(\zeta)=\zeta^\alp$ for $\zeta\in\mu_{p^\infty}$, and let $K$ be its fixed field.
Then $\Gal(\dbQ(\mu_{p^\infty})/K)=\langle\sig\rangle\isom\dbZ_p$.
Taking any homomorphic section of the restriction map $G_K\to\Gal(\dbQ(\mu_{p^\infty})/K)$ and a lift $\hat\sig$ of $\sig$, we obtain an extension $F$ of $K$ with $G_F=\langle\hat\sig\rangle\isom\dbZ_p$ and $\chi_F(\hat\sig)=\alp$.

In the remaining case where $\alp=1$, we have $\calZ^1=\dbZ_p\rtimes(1,1)=\dbZ_p\rtimes\calG_{\dbC}(p)$.
This pair can be realized, e.g., as $\calG_{\dbC((t))}(p)$, where $\dbC((t))$ is the field of power series in $t$ with coefficients in $\dbC$; See \S\ref{realizations of extensions} below.
\end{examples}

\section{Cohomology of extensions}
\label{section on cohomology of extension}
Given a profinite group $G$ we henceforth abbreviate $H^i(G)=H^i(G,\dbZ/p)$, $i\geq0$, for the $i$th profinite cohomology group of $G$ with its trivial action on $\dbZ/p$.
We identify $H^1(G)$ with the group of all continuous homomorphisms $\psi\colon G\to\dbZ/p$.
It is dual to $G/\Phi(G)$, where $\Phi(G)=G^p[G,G]$ is the Frattini subgroup.

Throughout this section we consider an extension 
\begin{equation}
\label{setting of an extension}
\bar\calG=(\bar G,\bar\theta), \quad \calG=(G,\theta)=A\rtimes\bar\calG,  \quad A=\dbZ_p^m, 
\end{equation}
where $m$ is a cardinal number.
We write $A$ additively.

There is a canonical $G$-action on $H^1(A)$ given by $({}^g\psi)(a)=\psi(g\inv ag)$ for $g\in G$, $\psi\in H^1(A)$, and $a\in A$.

\begin{lem}
\label{cohomological properties of extensions}
\begin{enumerate}
\item[(a)]
The restriction homomorphisms yield an isomorphism $H^1(G)\isom H^1(A)\oplus H^1(\bar G)$.
\item[(b)]
The action of $G$ on $H^1(A)$ is trivial.
\end{enumerate}
\end{lem}
\begin{proof}
(a) \quad
For $g\in G$ and $a\in A$ we have $[a,g]=a-\theta(g)a\in pA$.
Therefore 
\[
G/\Phi(G)=A/pA\times \bar G/\Phi(\bar G),
\]
and the assertion follows by duality.

\medskip

(b) \quad
Let $g\in G$, $\psi\in H^1(A)$ and $a\in A$.
As $\theta(g)\inv a-a\in pA$, one has
\[
({}^g\psi)(a)=\psi(g\inv ag)=\psi(\theta(g)\inv a)=\psi(a).
\qedhere
\]
\end{proof}

Following Wadsworth \cite{Wadsworth83}*{Th.\ 3.1 and Cor.\ 3.4} in the case $m=1$, we now describe the structure of $H^i(G)$ for any $i\geq0$.

First we take a totally ordered set $(L,\leq)$ of cardinality $m$.
Using Lemma \ref{cohomological properties of extensions}(a) we choose elements $\beta_l\in H^1(G)$, $l\in L$, whose restrictions to $H^1(A)$ form an $\dbF_p$-linear basis and whose restrictions to $H^1(\bar G)$ are zero.

\begin{prop}
\label{generalized Wadsworth formula}
One has 
\[
H^i(G)=\bigoplus_{j=0}^i\bigoplus_{{l_1\nek l_j\in L}\atop{l_1<\cdots<l_j}}\inf(H^{i-j}(\bar G))\cup\beta_{l_1}\cup\cdots\cup\beta_{l_j}.
\]
Moreover, for any list $l_1<\cdots<l_j$ in $L$ the map 
\[
H^{i-j}(\bar G)\to  H^i(G), \quad
\bar\varphi\mapsto \inf(\bar\varphi) \cup\beta_{l_1}\cup\cdots\cup\beta_{l_j},
\]
is injective.
\end{prop}
\begin{proof}
First assume that $m$ is finite.
We argue by induction on $m$.

The case $m=0$ is immediate.
For $m=1$ this is \cite{Wadsworth83}*{Th.\ 3.1 and Cor.\ 3.4}, where we note that the assumptions of these results are satisfied by Lemma \ref{cohomological properties of extensions}. 

Next, assume the assertion for $m\geq1$ and consider the extension $\dbZ_p\rtimes\calG\isom\dbZ_p^{m+1}\rtimes\bar\calG$.
Let $\beta_{l_0}$ correspond to the new copy of $\dbZ_p$ in this extension, so $l_0\not\in L$, and we may assume that $l<l_0$ for every $l\in L$.
By the case $m=1$ and the induction hypothesis, the underlying group $\dbZ_p\rtimes G$ of $\dbZ_p\rtimes\calG$ satisfies
\[
\begin{split}
H^i(\dbZ_p\rtimes G)&=\inf(H^i(G))\oplus(\inf(H^{i-1}(G))\cup\beta_{l_0}) \\
&=\bigoplus_{j=0}^i\bigoplus_{{l_1\nek l_j\in L}\atop{l_1<\cdots<l_j}}\inf(H^{i-j}(\bar G))\cup\beta_{l_1}\cup\cdots\cup \beta_{l_j} \\
&\qquad\oplus\bigoplus_{j=0}^{i-1}\bigoplus_{{l_1\nek l_j\in L}\atop{l_1<\cdots<l_j}}\inf(H^{i-1-j}(G))\cup\beta_{l_1}\cup\cdots\cup\beta_{l_j}\cup\beta_{l_0} \\
&=\bigoplus_{j=0}^i\bigoplus_{{l_1<\cdots<l_j}\atop{l_1\nek l_j\in L\cup\{l_0\}}}\inf(H^{i-j}(\bar G))\cup\beta_{l_1}\cup\cdots\cup\beta_{l_j}.
\end{split}
\]

The injectivity as in the last assertion of the proposition follows by straightforward induction.

The infinite case follows from the finite case by a limit argument. 
\end{proof}

\section{Valuations}
\label{section on valuations}
As we shall see in \S\ref{section on p-henselian valuations and extensions}, extensions of cyclotomic pro-$p$ pairs of Galois type arise naturally in valuation-theoretic settings. 
We begin by recalling some basic concepts and facts about valuations on fields, originating in the works of Krull, Neukirch, and others. 
Our primary reference is \cite{Efrat06}.

\subsection{Basic notions}
Let $F$ be a field and let $F^\times$ be its multiplicative group.
A (Krull) valuation on $F$ is a group epimorphism $v\colon F^\times\to\Gam_v$, where $\Gam_v=(\Gam_v,\geq)$, called the \textsl{value group} of $v$,  is an ordered abelian group, satisfying the \textsl{ultra-metric inequality} $v(a+b)\geq\min\{v(a),v(b)\}$ for every $a,b\in F^\times$ with $a\neq-b$.
Equivalently, 
\[
O_v=\{a\in F^\times\ |\ v(a)\geq0\}\cup\{0\}
\]
is a valuation ring in $F$.
Its unique maximal ideal is 
\[
\grm_v=\{a\in F^\times\ |\ v(a)>0\}\cup\{0\},
\]
and its \textsl{residue field} is $\bar F_v=O_v/\grm_v$.
We denote the residue of $a\in O_v$ in $\bar F_v$ by $\bar a$. 
Being an ordered abelian group, $\Gam_v$ is torsion-free.
The valuation $v$ is called \textsl{discrete} if $\Gam_v$ is isomorphic as an ordered abelian group to $\dbZ$ with its natural order.
Then an element $t$ of $F^\times$  such that $v(t)$ corresponds to $1\in\dbZ$ is called a \textsl{uniformizer} for $v$.
We do not distinguish between valuations with the same valuation ring.


Given valuations $v$ and $u$ on $F$, we say that $v$ is \textsl{finer} than $u$ (and $u$ is \textsl{coarser} than $v$) if $O_v\subseteq O_u$.
In this case the image of $O_v$ under the residue map $O_u\to\bar F_u$ is the valuation ring $O_w$ of a valuation $w=v/u$ on $\bar F_u$, called the \textsl{quotient valuation} of $v$ by $u$.
Conversely, given a valuation $u$ on $F$ and a valuation $w$ on $\bar F_u$, there is a unique valuation $v$ on $F$ which is finer than $u$ and such that $w=v/u$  \cite{Efrat06}*{Ch.\ 5}.
Then $\Gam_u=\Gam_v/\Gam_w$.
Since $\Gam_u$ is torsion-free, a snake lemma argument \cite{Efrat06}*{Prop.\ 5.2.1(h)} gives an exact sequence of abelian groups 
\begin{equation}
\label{exact sequence for coarsenings}
0\to\Gam_w/p\Gam_w\to\Gam_v/p\Gam_v\to\Gam_u/p\Gam_u\to 0.
\end{equation}

\subsection{Galois theory of valued fields}
\label{subsection on Galois theory of valued fields}
For every valuation $v$ on $F$ and every field extension $E$ of $F$ there is a valuation $\tilde v$ on $E$ which extends $v$ \cite{Efrat06}*{Cor.\ 14.1.2}.
There is a natural embedding $\Gam_v\hookrightarrow\Gam_{\tilde v}$ of ordered abelian groups, and a field embedding $\bar F_v\subseteq\bar E_{\tilde v}$.
When $E/F$ is algebraic, $\Gam_{\tilde v}/\Gam_v$ is a torsion group \cite{Efrat06}*{Cor.\ 14.2.3(a)}.

Now let $E$ be a Galois extension of $F$ and let $G=\Gal(E/F)$.
The \textsl{decomposition}, \textsl{inertia}, and \textsl{ramification} groups of a valued field extension $(E,\tilde v)/(F,v)$ are defined by
\[
\begin{split}
Z=Z(\tilde v/v)&=\{\sig\in\Gal(E/F)\ |\ \tilde v\circ\sig=\tilde v\},\\
T=T(\tilde v/v)&=\{\sig\in\Gal(E/F)\ |\ \forall a\in O_{\tilde v}:\ \tilde v(\sig(a)-a)>0\},\\
V=V(\tilde v/v)&=\{\sig\in T(\tilde v/v)\ |\ \forall a\in E^\times:\ \tilde v(1-\sig(a)/a)>0\},
\end{split}
\]
respectively  \cite{Efrat06}*{Ch.\ 15--16}.
One has $Z\trianglerighteq T\trianglerighteq V$.
The quotient $T/V$ is an abelian profinite group \cite{Efrat06}*{Cor.\ 16.2.7(d)}.
The group $V$ is pro-$p'$, if $p'=\Char\bar F_v>0$, and is trivial if $\Char\bar F_v=0$ \cite{Efrat06}*{Th.\ 16.2.3}.
There is an exact sequence
\begin{equation}
\label{exact sequence for decomposition group}
1\to T(\tilde v/v)\to Z(\tilde v/v)\xrightarrow{\Phi_v} \Aut(\bar E_{\tilde v}/\bar F_v)\to 1,
\end{equation} 
where $\Phi_v$ is induced by the residue map \cite{Efrat06}*{Prop.\ 16.1.3(a)(c) and Cor.\ 15.2.3(b)}.

Any other extension of $v$ to $E$ has the same valuation ring as $\tilde v\circ\sig$ for some $\sig\in G$ \cite{Efrat06}*{Th.\ 14.3.2}, and one has 
\begin{equation}
\label{conjugacy of decomposition groups}
Z(\tilde v\circ\sig/v)=\sig\inv Z(\tilde v/v)\sig
\end{equation}
\cite{Efrat06}*{Remark 15.1.1(b)}.

The valuation $v$ on $F$ is \textsl{Henselian relative to $E$} if it has a unique extension $u$ to $E$, or equivalently, $\Gal(E/F)=Z(\tilde v/v)$.
When $v$ is Henselian with respect to the algebraic closure $F^{\rm alg}$  (or equivalently, the separable closure $F^{\rm sep}$ \cite{Efrat06}*{Prop.\ 14.2.5}), we will say that $v$ is \textsl{Henselian}.
When $v$ is Henselian with respect to $F(p)$ we will say that $v$ is \textsl{$p$-Henselian}.
If $\Char\,\bar F_v\neq p$, then $v$ is $p$-Henselian if and only if $1+\grm_v\leq(F^\times)^p$ \cite{Efrat06}*{Prop.\ 18.2.4}.
If $v$ and $u$ are valuations on $F$ with $u$ coarser than $v$, then $v$ is Henselian relative to $E$ if and only if $u$ is Henselian relative to $E$  and $w=v/u$ is Henselian relative to $\bar F_u(p)$ \cite{Efrat06}*{Th.\ 20.3.2}.

\begin{lem}
\label{inertia groups}
Let $v$ be a $p$-Henselian valuation on a field $F$, let $u$ be a valuation on $F$ coarser than $v$, and let $w=v/u$ be the quotient valuation on $\bar F_u$.
Let $T(\tilde v/v)$, $T(\tilde u/u)$, and $T(\tilde w/w)$ be the inertia groups of $v,u,w$, chosen with respect to compatible extensions $\tilde v,\tilde u,\tilde w$ to $F(p)$ and $\bar F_v(p)$, respectively.
Then $T(\tilde w/w)\isom T(\tilde v/v)/T(\tilde u/u)$.
\end{lem}
\begin{proof}
We have a commutative diagram  (see \cite{Efrat06}*{\S20.3}):
\[
\xymatrix{
Z(\tilde u/u)\ar[r]^{\Phi_u}&G_{\bar F_u}(p)&Z(\tilde w/w)\ar@{^{(}->}[l]\ar[d]^{\Phi_w}\\
Z(\tilde v/v)\ar@{^{(}->}[u]\ar[rr]^{\Phi_v}&&G_{\bar F_v}(p)=G_{(\bar F_u)_w(p)}.
}
\]
Since $v$ is $p$-Henselian, so are $u$ and $w$.
Therefore the two inclusion maps are equalities, so $\Phi_v=\Phi_w\circ\Phi_u$.
A snake lemma argument shows that $\Ker(\Phi_w)\isom\Coker(\Ker(\Phi_u)\hookrightarrow \Ker(\Phi_v))$, as claimed. 
\end{proof}

\subsection{Explicit constructions of valuations}
\label{subsection on explicit constructions of valuations}
Consider subgroups  $S\leq H\leq F^\times$.
We will be interested in valuations $v$ on $F$ such that $1+\grm_v\leq S$ and $O_v^\times\leq H$. 
The set of all such valuations may be empty.
However if it is nonempty, then it contains a coarsest element, denoted in \cite{Efrat06}*{Sect.\ 11.1} by $v_{(2)}(S,H)$.
A deep and subtle construction of valuations due to Arason, Elman, Jacob and Ware (\cite{Ware81}, \cite{Jacob81}, \cite{ArasonElmanJacob87}) allows one to describe $v_{(2)}(S,H)$ explicitly as follows: 
Let
\[
\begin{split}
O^-(S,H)&=(1-S)\cap H,\\
O^+(S,H)&=\bigl\{x\in H\ \bigm|\ xO^-(S,H)\subseteq O^-(S,H)\bigr\},\\
O(S,H)&=O^-(S,H)\cup O^+(S,H).
\end{split}
\]
Then, for $v=v_{(2)}(S,H)$ one has
\[
O_v=O(S,H).
\]
See \cite{ArasonElmanJacob87}*{Prop.\ 3.2} or \cite{Efrat06}*{Lemma 11.3.3 and Lemma 11.3.4}.

In our applications we take $S=(F^\times)^p$, so we assume $(F^\times)^p\leq H$.
The following proposition is based on \cite{Efrat99c}*{Lemma 2.6 and Cor.\ 2.7} and \cite{HwangJacob95}*{Lemma 2.10}:

\begin{prop}
\label{existence of tame valuations}
Suppose that $\Char\,F\neq p$ and there exists a valuation $v$ on $F$ such that $1+\grm_v\leq(F^\times)^p$ and $O_v^\times\leq H$.
Then there exists such a valuation with $\Char\,\bar F_v\neq p$.
In particular, $v$ is $p$-Henselian. 
\end{prop}
\begin{proof}
After replacing $v$ by the valuation $v_{(2)}((F^\times)^p,H)$, we may assume that $O_v=O((F^\times)^p,H)$.

Assume that $\Char\,\bar F_v=p$, i.e., $p\in\grm_v$.
Then $\bar F_v^\times$ has no $p$-torsion, so the homomorphism $(1+\grm_v)/(1+\grm_v)^p\to O_v^\times/(O_v^\times)^p$ is injective.
The homomorphism $O_v^\times/(O_v^\times)^p\to F^\times/(F^\times)^p$ is also injective, and we conclude from $1+\grm_v\leq(F^\times)^p$ that 
\[
1+\grm_v=(1+\grm_v)^p;
\]
See \cite{Efrat06}*{Lemma 3.2.4}.
For every $a\in\grm_v$ we may therefore find $b\in \grm_v$ with $1+a=(1+b)^p\in 1+b^p+p\grm_v$.
Assuming further that $a\not\in p\grm_v$, we deduce using the ultra-metric inequality that $a\in b^p(1+\grm_v)\subseteq (F^\times)^p$.
Thus we have shown that 
\[
\grm_v\setminus p\grm_v\subseteq(F^\times)^p\leq H.
\]

In particular, $p\in\grm_v\setminus p\grm_v\subseteq H$, and therefore $p\inv\in H\setminus O((F^\times)^p,H)$, so
$p\inv\not\in O^+((F^\times)^p,H)$.
This means that there exists $c\in O^-((F^\times)^p,H)$ for which $p\inv c\not\in O^-((F^\times)^p,H)$.
Since $O_v^\times \leq H$ and $c\not\in H$ we have $c\in\grm_v$.
Moreover, $c\not\in H$ implies $p\inv c\not\in H$, so even $p\inv c\not\in O_v$.
Therefore $c\in\grm_v\setminus p\grm_v\subseteq H$, a contradiction.

Consequently, $\Char\,\bar F_v\neq p$.
Since $1+\grm_v\leq(F^\times)^p$, this implies that $v$ is $p$-Henselian.
\end{proof}

\section{$p$-Henselian valuations and extensions}
\label{section on p-henselian valuations and extensions}
\subsection{The tame case}
Extensions of cyclotomic pro-$p$ pairs arise naturally in the valuation-theoretic context as follows:

\begin{prop}
\label{extensions from tame valuations}
Let $F$ be a field containing a root of unity of order $p$.
Let $v$ be a $p$-Henselian valuation on $F$ such that $\Char\,\bar F_v\neq p$ and let $\tilde v$ be its unique extension to $F(p)$.
Let $m=\dim_{\dbF_p}(\Gam_v/p\Gam_v)$, considered as a cardinal number, and let $T=T(\tilde v/v)$.
Then $T\isom\dbZ_p^m$ and the residue map induces a decomposition
\[
\calG_F(p)=T\rtimes\calG_{\bar F_v}(p).
\]
\end{prop}
\begin{proof}
As $\Char\,\bar F_v\neq p$, the ramification group $V(\tilde v/v)$ is trivial.
Since $v$ is $p$-Henselian, $Z(\tilde v/v)=G_F(p)$.
Therefore \cite{Efrat06}*{Cor.\ 22.1.2 and Prop.\ 22.1.3} give the assertions at the pro-$p$ group level.
We further note that the cyclotomic characters $\chi_F$ and $\chi_{\bar F_v}$ are compatible with the residue map. 
\end{proof}

\subsection{The wild case}
The case of $p$-Henselian valued fields with residue characteristic $p$ is more subtle.
Under the additional assumption that $G_F(p)$ is finitely generated, the arithmetical structure of $F$ is then given by Theorem \ref{the wild case} below.
First we note the following lemma.

\begin{lem}
\label{not pro p}
Let $E$ be a field of characteristic $0$ and let $w$ be a discrete valuation on $E$ with finite residue field $\bar E_w$ of characteristic $p$.
Then $G_E$ is not pro-$p$.
\end{lem}
\begin{proof}
The completion $(\hat E,\hat w)$ of $(E,w)$ has the same value group and residue field as $(E,w)$, and therefore it is a finite extension of $\dbQ_p$ with its canonical ($p$-adic) valuation \cite{Serre67}*{Ch.\ II, Sect.\ 5}. 
It is a consequence of the Krasner--Ostrowski lemma that the algebraic closures satisfy $\hat E^{\rm alg}=\hat E E^{\rm alg}$ \cite{Efrat06}*{Cor.\ 18.5.3}.
Hence $G_{E^{\rm alg}\cap \hat E}=G_{\hat E}$, and this group embeds as a closed subgroup of $G_E$.
But by Galois theory of $p$-adic fields,  $G_{\hat E}$ is not pro-$p$.
Therefore $G_E$ also cannot be pro-$p$.
\end{proof}

The following fact was proved in \cite{Efrat99b}:

\begin{thm}
\label{the wild case}
Let $F$ be a field containing a root of unity of order $p$ and such that $G_F(p)$ is finitely generated.
Let $v$ be a $p$-Henselian valuation on $F$ with $\Char\,\bar F_v=p$.
Then there exists a coarsening $u$ of $v$ with residue field $E=\bar F_u$ of characteristic $0$ and quotient valuation $w=v/u$ on $E$, such that one of the following holds:
\begin{enumerate}
\item[(i)]
$w$ is discrete, $\bar F_v=\bar E_w$ is finite, the completion $(\hat E,\hat w)$ of $(E,w)$ is a finite extension of $\dbQ_p(\mu_p)$, and $\calG_{\hat E}(p)\isom \calG_E(p)$ via restriction; or
\item[(ii)]
$\Gam_w=p\Gam_w$ and $G_E(p)$ is a finitely generated free pro-$p$ group.
\end{enumerate}
\end{thm}

\begin{defin}
\label{pairs of p-adic type}
\rm
Cyclotomic pro-$p$ pairs of the form $\calG_{\hat E}(p)$, where $\hat E$ is a finite extension of $\dbQ_p(\mu_p)$, will be called cyclotomic pairs of \textsl{$p$-adic type}.
\end{defin}

\begin{cor}
\label{cor to wild case}
Let $(F,v)$ be a valued field as in Theorem \ref{the wild case}.
Then $\calG_F(p)\isom\dbZ_p^m\rtimes\bar\calG$ where $0\leq m<\infty$, and either:
\begin{enumerate}
\item[(i)]
$\bar\calG$ is a cyclotomic pro-$p$ pair of $p$-adic type; or
\item[(ii)]
The underlying group of $\bar\calG$ is a finitely generated free pro-$p$ group.
\end{enumerate}
Moreover, when the absolute Galois group $G_F$ is a pro-$p$ group, only case (ii) is possible.
\end{cor}
\begin{proof}
The field $F$ is $p$-Henselian also with respect to the coarsening $u$ of $v$ (see \S\ref{subsection on Galois theory of valued fields}).
Since $G_F(p)$ is finitely generated, $(F^\times:(F^\times)^p)<\infty$.
The valuation $v$ induces a group epimorphism $F^\times/(F^\times)^p\to\Gam_v/p\Gam_v$.
Since there is a group epimorphism $\Gam_v/p\Gam_v\to\Gam_u/p\Gam_u$ (see (\ref{exact sequence for coarsenings})), $m:=\dim_{\dbF_p}(\Gam_u/p\Gam_u)<\infty$.
By Proposition \ref{extensions from tame valuations}, $\calG_F(p)\isom\dbZ_p^m\rtimes\calG_E(p)$.
Hence we may take $\bar\calG=\calG_E(p)$.

Finally, suppose that $G_F$ is pro-$p$.
By (\ref{exact sequence for decomposition group}), $G_E=G_{\bar F_u}$ is an epimorphic image of $G_F$, whence is also pro-$p$.
By Lemma \ref{not pro p}, case (i) of Theorem \ref{the wild case} is impossible, so we are in case (ii).
\end{proof}

\section{Realization of the basic operations over fields}
\label{section on reaization of the operations over fields}
In this section we show that the two basic operations in the category of cyclotomic pro-$p$ pairs -- extensions and free products -- are also well defined when restricted to Galois cyclotomic pairs $\calG_F(p)$ of Galois type (Definition \ref{pairs of Galois type}).
In addition, they are well defined even when we restrict to the sub-family of pairs $\calG_F(p)$ for which the absolute Galois group $G_F$ is pro-$p$.

\subsection{Realization of extensions}
\begin{prop}
\label{realizations of extensions}
Let $\bar F$ be a field containing a root of unity of order $p$ and let $m\leq m'$ be cardinal numbers.
Then:
\begin{enumerate}
\item[(a)]
There exists a Henselian valued field $(F,v)$ containing $\bar F$ with
\[
\calG_F(p)\isom\dbZ_p^m\rtimes\calG_{\bar F}(p), \ 
\trdeg_{\bar F}F=m', \  \dim_{\dbF_p}(\Gam_v/p\Gam_v)=m, 
\]
and $\bar F_v/\bar F$ a purely inseparable field extension.
\item[(b)]
If in addition $G_{\bar F}$ is a pro-$p$ group, then we can find $F$ as in (a) with $G_F$ a pro-$p$ group.
\end{enumerate}
\end{prop}
\begin{proof}
(a) \quad
Let $\dbZ_{(p)}$ be the localization of $\dbZ$ at the ideal $p\dbZ$, and consider it as an ordered abelian group with respect to the order induced from $\dbQ$.
Take a well-ordered set $I'$ of cardinality $m'$ and a subset $I$ of $I'$ of cardinality $m$.
Let $\Gam=\dbZ_{(p)}^I\times\dbQ^{I'\setminus I}$, considered as a totally ordered abelian group with respect to the lexicographic order.
Then $\dim_{\dbF_p}(\Gam/p\Gam)=m$.

Next let $\bar F((\Gam))$ be the field of formal power series in the variable $t$ with coefficients in $\bar F$ and exponents in $\Gam$ and with well-ordered support \cite{Efrat06}*{\S 2.8}.
The function of taking the minimum of the support is a Henselian valuation $\hat v$ on $F((\Gam))$ with value group $\Gam$ and residue field $\bar F$ \cite{Efrat06}*{Example 4.2.1 and Cor.\ 18.4.2}. 

Let $\bar F(\Gam)$ be the field of generalized rational functions over $\bar F$ with exponents in $\Gam$, that is, $\bar F(\Gam)$ is the subfield of $\bar F((\Gam))$ generated by all formal power series with finite support \cite{Efrat06}*{\S2.9}. 
It has the same value group $\Gam$ and residue field $\bar F$ as $\bar F((\Gam))$.
Let $F$ be the relative algebraic closure of $\bar F(\Gam)$ in $\bar F((\Gam))$.
The restriction $v$ of $\hat v$ to $F$ then satisfies $\Gam_v=\Gam$ and $\bar F_v=\bar F$.
Since $(\bar F((\Gam)),\hat v)$ is Henselian, $(F,v)$ is also Henselian \cite{Efrat06}*{Prop.\ 15.3.3}.

We identify elements of $\bar F$ as constant rational functions.
In this way $\bar F$ is a relatively algebraically closed subfield of $\bar F(\Gam)$ with relative transcendence degree $\dim_\dbQ(\Gam\tensor\dbQ)=m'$ \cite{Efrat06}*{Prop.\ 2.9.1}.

By assumption, $\Char\,\bar F\neq p$.
By Proposition \ref{extensions from tame valuations}, $\calG_F(p)\isom\dbZ_p^m\rtimes\calG_{\bar F}(p)$.

\medskip

(b) \quad
In the above construction, let $\tilde v$ be the unique extension of $v$ to $F^{\rm sep}$.
It has value group $\Gam\tensor_\dbZ\dbQ=\dbQ^{I'}$.
In the notation of \S\ref{subsection on Galois theory of valued fields}
\[
T(\tilde v/v)/V(\tilde v/v)\isom \Hom(\dbQ^{I'}/(\dbZ_{(p)}^I\times\dbQ^{I'\setminus I}),\mu_{p^\infty})=\Hom(\tfrac1{p^\infty}\dbZ,\mu_{p^\infty})^I
\]
\cite{Efrat06}*{Th.\ 16.2.6}, which is a pro-$p$ group.
From (\ref{exact sequence for decomposition group}) we obtain an exact sequence
\[
1\to T(\tilde v/v)/V(\tilde v/v)\to G_F/V(u/v)\to G_{\bar F}
\to1.
\]
Since $G_{\bar F}$ is pro-$p$, we deduce that $G_F/V(\tilde v/v)$ is also pro-$p$.

Moreover, the restriction homomorphism $G_F\to G_F/V(\tilde v/v)$ splits \cite{Efrat06}*{Th.\ 22.2.1}.
Let $F'$ be the corresponding separable algebraic extension of $F$ and $v'=\tilde v|_{F'}$.
Thus $G_{F'}\isom  G_F/V(\tilde v/v)$.
We recall that $V(\tilde v/v)$ is either trivial or a pro-$p'$ group if $p'=\Char\,\bar F>0$.
In the latter case, by \cite{KuhlmannPankRoquette86}*{Lemma 4.2}, $\Gam_{v'}/\Gam_v$ has only $p'$-torsion and  $(\overline{F'})_{v'}/\bar F_v$ is purely inseparable.
Since $\Gam_v=\Gam$ is $p'$-divisible, we deduce that $\Gam_{v'}/p=\Gam/p$ \cite{Efrat06}*{Lemma 1.1.4(c)}.
Furthermore, $G_{(\overline{F'})_{v'}}\isom G_{\bar F_v}=G_{\bar F}$. 
Consequently, $(F',v')$ is a valued field as required. 
\end{proof}

\subsection{Realization of free  products}
Part (b) of the following proposition is essentially \cite{EfratHaran94}*{Prop.\ 1.3}.

\begin{prop}
\label{realizations of free products}
Let $F_1\nek F_n$ be fields which contain a root of unity of order $p$.
Then:
\begin{enumerate}
\item[(a)]
There exists a field $F$ containing a root of unity of order $p$ such that
\[
\calG_F(p)\isom\calG_{F_1}(p)*\cdots*\calG_{F_n}(p).
\]
\item[(b)]
If in addition $G_{F_1}\nek G_{F_n}$ are pro-$p$, then we can find $F$ as in (a) with $G_F$ a pro-$p$ group.
\end{enumerate}
\end{prop}
\begin{proof}
(a) \quad
A construction using the Witt vectors functor shows that for every field $\bar K$ there exists a field $K$ of characteristic $0$ such that $G_K\isom G_{\bar K}$, where moreover $K$ is Henselian with respect to a valuation $v$ with residue field containing $\bar K$ \cite{Efrat95}*{Prop.\ 4.7}.
In particular, if $\bar K$ contains a root of unity of order $p$, then the same holds for $K$.
Hence we may assume that $\Char\,F_i=0$ for each $i$.

By Proposition \ref{realizations of extensions} (with $m=0$), we may assume that $F_1\nek F_n$ have the same transcendence degree over $\dbQ$.
After identifying transcendence bases, we may assume further that $F_1\nek F_n$ are algebraic extensions of a single field $F_0$ containing a root of unity of order $p$.

Next, let $t$ be a transcendental element over $F_0$ and set $E_0=F_0(t)$.
Choose distinct elements $a_1\nek a_n\in F_0$ (e.g., $a_i=i$). 
For each $1\leq i\leq n$ we consider $F_i((\dbQ))$ as the field of formal power series in the variable $t+a_i$ with exponents in $\dbQ$ and coefficients in $F_i$.
Let $w_i$ be its canonical Henselian valuation.
Then $E_0=F_0(t+a_i)$ embeds in $F_i((\dbQ))$ in the natural way.

Let $E_i=E_0(p)\cap F_i((\dbQ))$, and let $u_i$ be the restriction of $w_i$ to $E_i$.
It also has residue field $F_i$, and it contains $(t+a_i)^{1/p^k}$ for every $k\geq1$, so $\Gam_{u_i}=p\Gam_{u_i}$.
Moreover, it is Henselian relative to $E_0(p)$, that is, $p$-Henselian \cite{Efrat06}*{Prop.\ 15.3.3}.
By Proposition \ref{extensions from tame valuations}, $\calG_{E_i}(p)\isom\calG_{F_i}(p)$ for each $1\leq i\leq n$ via the residue maps.

Take $E=E_1\cap\cdots\cap E_n$.
Then $E_0\subseteq E\subseteq E_i\subseteq E_0(p)$ for every $1\leq i\leq n$.
Let $v_i$ be the restriction of $u_i$ to $E$.
Then $v_i(t+a_i)=1$ and $v_i(t+a_j)=0$ for $j\neq i$.
Thus $v_1\nek v_n$ are distinct.
Moreover, their value groups are ordered subgroup of $\dbQ$, so they have rank $1$ \cite{Efrat06}*{Th.\ 2.5.2}.
It follows that they induce distinct topologies on $E$ \cite{Efrat06}*{Example 10.1.1 and Th.\ 10.1.7}.
By \cite{JacobWadsworth86}*{Th.\ 4.3} or \cite{Efrat97b}*{Prop.\ 4.3},
\[
G_E(p)=G_{E_1}(p)*_p\cdots*_pG_{E_n}(p)\isom
G_{F_1}(p)*_p\cdots*_pG_{F_n}(p).
\]
Moreover, this isomorphism is compatible with the cyclotomic characters, so it also holds for the Galois cyclotomic pro-$p$ pairs.

\medskip

(b) \quad
Assume that $G_{F_1}\nek G_{F_n}$ are pro-$p$.
As in (a) we may assume that $F_1\nek F_n$ have characteristic $0$.
Then the absolute Galois group of $F_i((\dbQ))$ is also pro-$p$ for every $i$.
We amend the previous construction by taking $E_i$ to be the relative algebraic closure of $E_0$ in $F_i((\dbQ))$ with the restricted valuation.
Its absolute Galois group is pro-$p$ as well.
After replacing $E_i$ by an $E_0$-isomorphic copy, we may assume that $G_{E_i}$ is contained in some given pro-$p$-Sylow subgroup of $G_{E_0}$.
It follows that $G_E$ is pro-$p$, and the rest of the argument continues as in (a).
\end{proof}

\section{The Elementary Type Conjecture}
\label{section on the elementary type conjecture}
We first state the conjecture in its arithmetical form, as in \cite{Efrat95} and \cite{Efrat97a}.
As before, we assume that $F$ is a field containing a root of unity of order $p$.

\begin{conj}[ETC - arithmetical form]
\label{ETC arithmetical form}
Suppose that $G_F(p)$ is finitely generated as a pro-$p$ group.
Then there is a free pro-$p$ product decomposition $G_F(p)=G_1*_p\cdots*_pG_n$, where for each $1\leq i\leq n$ one of the following holds:
\begin{enumerate}
\item[(i)]
$G_i\isom\dbZ_p$;
\item[(ii)]
$G_i$ is the decomposition group of some extension $\tilde v$ of a valuation $v$ on $F$ to $F(p)$ with nontrivial inertia group;
\item[(iii)]
$p=2$ and $G_i\isom\dbZ/2$.
\end{enumerate}
\end{conj}

\begin{rems}
\label{remarks on the arithmetical ETC}
\rm
(1) \quad
We allow the case $n=0$, in which $G_F(p)=1$.

\medskip

(2) \quad
In (ii) there is no dependence on the choice of the extension $\tilde v$ of $v$.
Indeed, changing the extension amounts to replacing the decomposition group $Z(\tilde v/v)$ by a conjugate subgroup in $G_F(p)$ (see (\ref{conjugacy of decomposition groups})).
Furthermore, for closed subgroups $G_1\nek G_n$ of a general pro-$p$ group $G$  and for every $\sig_1\nek\sig_n\in G$, if $G=G_1*_p\cdots*_pG_n$, then also $G=G^{\sig_1}*_p\cdots*_pG_n^{\sig_n}$ \cite{Ribes91}*{Lemma 4.1}.

\medskip

(3) \quad
In case (iii), the fixed field $E_i$ of $G_i$ in $F(2)$ is a \textsl{Euclidean closure} (i.e., a relative real closure in $F(2)$) of $F$ with respect to some ordering on $F$; See \cite{Becker74} or \cite{Efrat06}*{\S19.2}.

\medskip

(4) \quad
On the other hand, free factors of type (i) need not arise from arithmetic objects.
For example, since $\dbZ_p$ is a projective profinite group it can be realized as the absolute Galois group of a pseudo algebraically closed (PAC) field \cite{FriedJarden23}*{Th.\ 23.1.1}.
Such fields carry no nontrivial valuations with a nontrivial decomposition group \cite{FriedJarden23}*{Cor.\ 11.5.5}, nor do they carry an ordering, since this would yield a $2$-torsion in the absolute Galois group.   
\end{rems}

We recall that the \textsl{rank} $\rank(G)$ of a pro-$p$ group $G$ is the minimal number of generators of $G$.
Alternatively, $\rank(G)=\dim_{\dbF_p}H^1(G)$ \cite{NeukirchSchmidtWingberg}*{Prop.\ 3.9.1}.
For a cyclotomic pro-$p$ pair $\calG=(G,\theta)$ we define $\rank(\calG)=\rank(G)$. 

\begin{thm}
\label{decomposition theorem}
Suppose that Conjecture \ref{ETC arithmetical form} holds for $F$.
Then 
\[
\calG_F(p)\isom\calG_1*\cdots*\calG_n,
\]
where for each $1\leq i\leq n$ either:
\begin{enumerate}
\item[(1)]
$\calG_i=\calZ^\alp$ for some $\alp\in\dbZ_p^{\times,1}$;
\item[(2)]
$p=2$ and $\calG_i=\calE$;
\item[(3)]
$\calG_i$ is a cyclotomic pro-$p$ pair of $p$-adic type;
\item[(4)]
$\calG_i=\dbZ_p^m\rtimes\bar\calG_i$, where $1\leq m<\infty$ and $\bar\calG_i$ is a cyclotomic pro-$p$ pair of Galois type with $\rank(\bar\calG_i)<\rank(\calG_i)$.
\end{enumerate}
\end{thm}
\begin{proof}
Consider the decomposition as in Conjecture \ref{ETC arithmetical form}.
For each $1\leq i\leq n$ let $E_i$ be the fixed field of $G_i$ in $F(p)$, and let $\calG_i=\calG_{E_i}(p)$.
Thus $\calG_F(p)=\calG_1*\cdots*\calG_n$.
Each $\calG_i$ is an epimorphic image of $\calG_F(p)$, and in particular its underlying group is finitely generated.

When $G_i\isom\dbZ_p$ we clearly have (1).

When $p=2$ and $G_i\isom\dbZ/2$, the field $E_i$ is Euclidean (Remark \ref{remarks on the arithmetical ETC}(3)).
In particular $F(2)=E_i(\sqrt{-1})$.
Hence the pro-$2$ cyclotomic character $\chi_{E_i}$ is nontrivial, and we get (2). 

It therefore remains to consider the case where $G_i$ is a decomposition group of some extension $\tilde v$ of a valuation $v$ on $F$ to $F(p)$ for which the inertia group $T(\tilde v/v)$ is nontrivial. 
The restriction of $\tilde v$ to $E_i$ is a Henselian valuation relative to $F(p)=E_i(p)$, i.e., it is $p$-Henselian.
It has the same value group $\Gam_v$ and residue field $\bar F_v$ as $v$ \cite{Efrat06}*{Th.\ 15.2.2}.

First consider the tame case $\Char\bar F_v\neq p$.
Therefore, by Proposition \ref{extensions from tame valuations}, $T(\tilde v/v)\isom\dbZ_p^m$, where $m=\dim_{\dbF_p}(\Gam_v/p\Gam_v)$, and 
\[
\calG_i=\calG_{E_i}(p)=T(\tilde v/v)\rtimes \calG_{\bar F_v}(p)\isom\dbZ_p^m\rtimes\calG_{\bar F_v}(p).
\]
Since $G_F(p)$ is finitely generated, the group $F^\times/(F^\times)^p$ is finite, and therefore so is its quotient $\Gam_v/p\Gam_v$.
Since $T(\tilde v/v)$ is nontrivial, we therefore have $1\leq m<\infty$.
By Lemma \ref{cohomological properties of extensions}(a), 
\[
\rank(G_i)=\dim_{\dbF_p}H^1(G_i)=\dim_{\dbF_p}H^1(G_{\bar F_v}(p))+m>\rank(G_{\bar F_v}(p)).
\]
Moreover, $G_i$ is a quotient of $G_F(p)$, so 
\[
\rank(G_F(p))\geq\rank(G_i)>\rank(G_{\bar F_v}(p)).
\]
Taking $\bar F=\bar F_v$, we obtain (4) in this case.

Finally we consider the wild case, where $\Char\,\bar F_v=p$.
Then, by Corollary \ref{cor to wild case}, $\calG_i=\dbZ_p^m\rtimes\bar\calG_i$, where $0\leq m<\infty$ and $\bar\calG_i$ is either a cyclotomic pro-$p$ pair of $p$-adic type or its underlying group is a finitely generated free pro-$p$ group.
Note that cyclotomic pairs whose underlying group is a finitely generated free pro-$p$ groups are of Galois type, by Proposition \ref{realizations of free products} and Example \ref{examples of realizable pairs}(4).
When $m\geq1$ we are therefore in case (4).

Hence we may assume that $m=0$.
If $\calG_i=\bar\calG_i$ is of $p$-adic type, then we are in case (3).
If the underlying group of $\calG_i=\bar\calG_i$ is a finitely generated free pro-$p$ group, then $\calG_i$ is a free product of pairs of type (1), so we are done by increasing $n$.
\end{proof}

\begin{defin}
\label{Definition of ETp}
\rm
The class $\ET_p$ of cyclotomic pro-$p$ pairs of \textsl{elementary type} is the minimal class of cyclotomic pro-$p$ pairs (up to isomorphism) containing the `building blocks':
\begin{enumerate}
\item[(1)]
the trivial pair $(1,1)$; 
\item[(2)]
the pairs $\calZ^\alp$, where $\alp\in\dbZ_p^{\times,1}$;
\item[(3)]
$\calE$, if $p=2$;
\item[(4)]
the cyclotomic pro-$p$ pairs of $p$-adic type;
\end{enumerate}
and which is close under the operations of extensions $\bar\calG\mapsto\dbZ_p\rtimes\bar\calG$ and free products.
\end{defin}

By induction, pairs in $\ET_p$ have finite rank.

\begin{conj}[ETC - Lego form]
\label{ETC Lego version}
Every cyclotomic pro-$p$ pair of Galois type and of finite rank is in $\ET_p$.
\end{conj}

In view of Theorem \ref{decomposition theorem}, we have:

\begin{cor}
\label{cor to ETC}
Conjecture \ref{ETC arithmetical form} implies Conjecture \ref{ETC Lego version}.
\end{cor}

The following converse of Conjecture \ref{ETC Lego version} holds unconditionally:

\begin{prop}
Every cyclotomic pro-$p$ pair $\calG$ in $\ET_p$ is of Galois type.
\end{prop}
\begin{proof}
For building blocks $\calZ^\alp$ this is Example \ref{examples of realizable pairs}(4).
When $p=2$ we have $\calE=\calG_\dbR(2)$.
For $p$-adic pairs this holds by definition.
The operations of extensions and free products are realizable Galois-theoretically by Proposition \ref{realizations of extensions} and Proposition \ref{realizations of free products}, respectively.
\end{proof}

\section{Variants}
\label{section on variants}
Restricting to subfamilies of fields gives variants of the Elementary Type Conjecture, which we now describe.

\subsection{The absolute case}
Here we restrict to fields $F$ whose absolute Galois group $G_F$ is pro-$p$.

\begin{thm}
\label{absolute decomposition theorem}
If Conjecture \ref{ETC arithmetical form} holds for $F$ and $G_F$ is pro-$p$, then 
\[
\calG_F\isom\calG_1*\cdots*\calG_n,
\]
where for each $1\leq i\leq n$ either:
\begin{enumerate}
\item[(i)]
$\calG_i=\calZ^\alp$ for some $\alp\in\dbZ_p^{\times,1}$;
\item[(ii)]
$p=2$ and $\calG_i=\calE$;
\item[(iii)]
$\calG_i=\dbZ_p^m\rtimes\bar\calG_i$, where $m\geq1$ and  $\bar\calG_i=\calG_{\bar F}(p)$ is the cyclotomic pro-$p$ pair of some field $\bar F$, with $G_{\bar F}$ a pro-$p$ group of rank strictly smaller than that of $G_F$.
\end{enumerate}
\end{thm}
\begin{proof}
We argue as in the proof of Theorem \ref{decomposition theorem} with the following two amendments:

In the tame case we note that $G_{E_i}$ is pro-$p$, and therefore its quotient $G_{\bar F_v}$ is also pro-$p$ (see (\ref{exact sequence for decomposition group})).

In the wild case, Theorem \ref{the wild case} shows that $\calG_i=\dbZ_p^m\rtimes\bar\calG_i$, where $m\geq0$ and the underlying group of $\bar\calG_i$ is a finitely generated free pro-$p$ group.
Thus in this case $p$-adic type pairs do not arise in the decomposition.
\end{proof}

\begin{defin}
\rm
The class $\ET_p^{\rm abs}$ of cyclotomic pro-$p$ pairs of \textsl{absolute elementary type} is the minimal class of cyclotomic pro-$p$ pairs (up to isomorphism) which contains the `building blocks' $\calZ^\alp$, $\alp\in\dbZ_p^{\times,1}$, and $\calE$, if $p=2$,
and which is closed under the operations of extensions $\bar\calG\mapsto\dbZ_p\rtimes\bar\calG$ and free products.
\end{defin}

Theorem \ref{absolute decomposition theorem} gives:

\begin{cor}
\label{cor to absolute ETC}
Assume Conjecture \ref{ETC arithmetical form}.
Then for every field $F$ such that $G_F$ is pro-$p$ and finitely generated we have $\calG_F(p)\in\ET_p^{\rm abs}$.
\end{cor}

\begin{rem}
\label{converse of ETC abs}
\rm
Again, the converse of Corollary \label{cor to absolute ETC} holds unconditionally.
Indeed, for building blocks $\calZ^\alp$ this was shown in Example \ref{examples of realizable pairs}(4).
In the case $\calG=\calE$, $p=2$, we have $\calG=\calG_\dbR(2)$ and $G_\dbR\isom\dbZ/2$.

The operations of extensions and free products are realizable for pro-$p$ absolute Galois groups by  Proposition \ref{realizations of extensions}(b) and \ref{realizations of free products}(b), respectively.
\end{rem}

\subsection{Fields containing $\mu_{p^\infty}$}
Here we restrict to fields $F$ of characteristic $\neq p$  which contain the group $\mu_{p^\infty}$ of all roots of unity of $p$-power order.
Equivalently, the cyclotomic character $\chi_F$ is trivial, so a description of $\calG_F(p)$ amounts to describing $G_F(p)$.

\begin{thm}
\label{decomposition theorem for fields containing all roots of unity}
Suppose that Conjecture \ref{ETC arithmetical form} holds for $F$ where $\Char\,F\neq p$ and $\mu_{p^\infty}\subseteq F$.
Then 
\[
G_F(p)\isom G_1*_p\cdots*_pG_n,
\]
where for each $1\leq i\leq n$ either:
\begin{enumerate}
\item[(i)]
$G_i=\dbZ_p$;
\item[(ii)]
$G_i=\dbZ_p^m\times G_{\bar F_i}(p)$, where $1\leq m<\infty$ and $\bar F_i$ is a field such that $\Char\,\bar F_i\neq p$, $\mu_{p^\infty}\subseteq \bar F_i$, and the rank of $G_{\bar F_i}(p)$ is strictly smaller than that of $G_F(p)$.
\end{enumerate}
\end{thm}
\begin{proof}
Take the decomposition as in Theorem \ref{decomposition theorem}.
Then $\calG_i=(G_{E_i}(p),1)$ for some subextensions $F(\mu_{p^\infty})\subseteq E_1\nek E_n\subseteq F(p)$.
Since the characters of $\calE$ (when $p=2$) and of $p$-adic pairs are nontrivial, only cases (i) and (iv) of this theorem are possible, yielding the two possibilities above. 
\end{proof}

\begin{defin}
\rm
The class $\ET_p^{\infty}$ of cyclotomic pro-$p$ pairs of elementary type with trivial character is the minimal class of pairs (up to isomorphism) which contains
$\calZ^1$ and which is closed under the operations of free products and extensions $(\bar G,1)\mapsto(\dbZ_p\times\bar G,1)$.
\end{defin}

Theorem \ref{decomposition theorem for fields containing all roots of unity} gives:

\begin{cor}
\label{cor to infinite ETC}
Assume Conjecture \ref{ETC arithmetical form}.
Then every cyclotomic pro-$p$ pair of Galois type with trivial character and of finite rank is contained in $\ET_p^{\infty}$.
\end{cor}

\begin{rem}
\rm
Note that $\ET_p^{\infty}\subseteq \ET_p^{\rm abs}$.
Therefore, by Remark \ref{converse of ETC abs}, the converse of Corollary \label{cor to infinite ETC} holds unconditionally.
\end{rem}

\subsection{The Pythagorean case}
\label{subsection on the Pythagorean case}
Here we take $p=2$ and assume that the field $F$ satisfies $\Char\,F\neq2$.
One says that $F$ is \textsl{Pythagorean} if $F^2=F^2+F^2$.
Equivalently, $G_F(2)$ is generated by elements of order $2$ (\cite{Becker74}, \cite{Efrat06}*{Prop.\ 19.2.5 and Th.\ 19.2.10}).
One has $\Char\,F=0$ unless $G_F(2)=1$.
The following variant of Theorem \ref{decomposition theorem} was proved (unconditionally) by Jacob \cite{Jacob81} and Min\'a\v c \cite{Minac86}.
Generalizations to non-finitely generated Pythagorean fields were given in \cite{Efrat93} and \cite{EfratHaran94}.

\begin{thm}[Jacob, Min\'a\v c]
\label{decomposition theorem for Pythagorean fields}
Suppose that $F$ is Pythagorean and $(F^\times:(F^\times)^2)<\infty$.
Then 
\[
\calG_F(2)\isom\calG_1*\cdots*\calG_n,
\]
where for each $1\leq i\leq n$ either:
\begin{enumerate}
\item[(i)]
$\calG_i\isom\calE$;
\item[(ii)]
$\calG_i=\dbZ_2^m\rtimes \bar\calG_i$, where $1\leq m<\infty$ and $\bar\calG_i=\calG_{\bar F_i}(2)$ for a Pythagorean field $\bar F_i$ such that the rank of $G_{\bar F_i}(2)$ is strictly smaller than that of $G_F(2)$.
\end{enumerate}
\end{thm}

\begin{defin}
\rm
The class $\ET^{\rm Pyth}$ of cyclotomic pro-$2$ pairs of \textsl{Pythagorean elementary type} is the minimal class of pairs (up to isomorphism) which contains  $\calE$
and which is closed under the operations of extensions $\bar\calG\mapsto\dbZ_2\rtimes\bar\calG$ and free products.
\end{defin}

\begin{cor}
\label{cor to Pythagorean ETC}
For every field $F$ of characteristic $0$, the following two conditions are equivalent:
\begin{enumerate}
\item[(a)]
$F$ is Pythagorean and $(F^\times:(F^\times)^2)<\infty$;
\item[(b)] 
$\calG_F(2)\in\ET^{\rm Pyth}$.
\end{enumerate}
\end{cor}
\begin{proof}
The implication (a)$\Rightarrow$(b) is by Theorem \ref{decomposition theorem for Pythagorean fields}.
For the implication (b)$\Rightarrow$(a) see \cite{EfratHaran94}*{Th.\ 3.1}.
\end{proof}

\section{Evidence}
\label{section on evidence}
As discussed in the Introduction, the Elementary Type Conjecture emerged from similar conjectures in the context of quadratic forms.
These roughly correspond to the Lego version (Conjecture \ref{ETC Lego version}) in the case $p=2$, by the methods of \cite{JacobWare89} and \cite{JacobWare91}.
They were proved in many cases, including:
\begin{enumerate}
\item[(1)]
When $(F^\times:(F^\times)^2)\leq8$, i.e., $G_F(2)$ is generated as a pro-$2$ group by $\leq3$ elements, by Kula, Szczepanik, and Szymiczek \cite{KulaSzczepanikSzymiczek79};
\item[(2)]
When  $(F^\times:(F^\times)^2)=16$, by L.\ Berman (unpublished) and independently Szczepanik  \cite{Szczepanik85};
\item[(3)]
When  $(F^\times:(F^\times)^2)=32$, by Carson and Marshall with the aid of a computer \cite{CarsonMarshall82};
\item[(4)]
When $(F^\times:(F^\times)^2)=64$ or $=128$, by Lorenz and Sch\"onert, again with the aid of a computer \cite{LorenzSchonert25};
\item[(5)]
When the number of 2-fold Pfister forms over $F$ is $\leq 12$ (Kaplansky \cite{Kaplansky69}, Cordes \cite{Cordes82}, Kula \cite{Kula91}).
This is the same as the number of different cup products in $H^2(G_F(2))$, implying the conjecture when  $G_F(2)$ has $\leq3$ defining relations (see \cite{NeukirchSchmidtWingberg}*{Cor.\ 3.9.5}).
\end{enumerate}

As for the stronger arithmetical form of the conjecture (Conjecture \ref{ETC arithmetical form}), it is natural to prove it for fields with good local--global behaviour, such as for fields for which the Hasse principle for the Brauer groups holds.
In particular, it was proved for the following classes of fields: 
\begin{exam}
\rm
\begin{enumerate}
\item[(1)]
Algebraic extensions $F$ of global fields such that $\mu_p\subseteq F$ and $G_F(p)$ is finitely generated (\cite{Efrat97a}, \cite{Efrat99a});
\item[(2)]
Fields $F$ of transcendence degree $\leq1$ over a local field such that $\mu_p\subseteq F$ and $G_F(p)$ is finitely generated (\cite{Efrat00}, \cite{KrullNeukirch71});
\item[(3)]
Fields $F$ of transcendence degree $\leq1$ over a pseudo algebraically closed (PAC) field such that $\mu_p\subseteq F$ and $G_F(p)$ is finitely generated \cite{Efrat01}.
\end{enumerate}
\end{exam}

Another class of examples is that of fields $F$, containing as usual a root of unity of order $p$, such that $G_F(p)$ is `arithmetically full' in the following sense.
We consider subextensions $F\subseteq E\subseteq F(p)$ such that either:
\begin{enumerate}
\item[(i)]
$E$ is $p$-Henselian with respect to a valuation $v$ with $\Gam_v\neq p\Gam_v$ and $\Char\, \bar E_v\neq p$; or
\item[(ii)]
$p=2$ and $E$ is Euclidean.
\end{enumerate}

\begin{thm}[\cite{Efrat97b}*{Th.\ 4.5 and Lemma 1.2}]
Suppose that $G_F(p)$ is generated by subgroups $G_{E_i}(p)$, $i=1,2\nek n$, with $E_1\nek E_n$ as in (i) or (ii).
Moreover, suppose that $n$ is the minimal positive integer for which such generating subgroups exist.
Then 
\[
G_F(p)=G_{E_1}(p)*_p\cdots*_pG_{E_n}(p).
\]
Moreover, in case (i) $G_{E_i}(p)$ is a decomposition field of $v|_F$ in $F(p)$.
\end{thm}

In view of Proposition \ref{extensions from tame valuations}, this implies:

\begin{cor}
\label{full fields}
Suppose that $F$ is the intersection of finitely many intermediate fields $F\subseteq E\subseteq F(p)$ satisfying (i) or (ii).
Then Conjecture \ref{ETC arithmetical form} holds for $F$.
\end{cor}

An important special case is when $p=2$ and $F$ is a Pythagorean field with $(F^\times:(F^\times)^2)<\infty$.
Then $F$ is the intersection of finitely many Euclidean closures, so Corollary \ref{full fields} recovers in this case the results of \S\ref{subsection on the Pythagorean case}.

\section{Demu\v skin groups}
\label{section on Demushkin groups}
The arithmetical form of the Elementary Type Conjecture implies other long-standing conjectures concerning the realization of pro-$p$ Demu\v skin groups as maximal pro-$p$ Galois groups.
This will be elaborated in this and the next sections.
In fact, as we shall see in \S\ref{section on equivalence of conjectures}, understanding the realization of these groups Galois-theoretically constitutes the exact difference between the arithmetical and Lego forms of the Elementary Type Conjecture.

\subsection{Definition and presentations} 
Recall that the rank of a pro-$p$ group $G$ is $\rank(G)=\dim_{\dbF_p}H^1(G)$ (\S\ref{section on the elementary type conjecture}).

\begin{defin} [\cite{Labute67}]
\label{Demuskin groups}
\rm
A pro-$p$ group $G$ is said to be \textsl{Demu\v skin} if the following conditions are satisfied:
\begin{enumerate}
\item[(i)]
$n=\rank(G)<\infty$;
\item[(ii)]
$\dim_{\dbF_p}H^2(G)=1$;
\item[(iii)]
The cup product $\cup\colon H^1(G)\times H^1(G)\to H^2(G)$ is a non-degenerate bilinear map.
\end{enumerate}
\end{defin}

\begin{exam}
\rm
A finite group $G$ is a pro-$p$ Demu\v skin group if and only if $G=\dbZ/2$, $p=2$ \cite{NeukirchSchmidtWingberg}*{Prop.\ 3.9.10}.
Furthermore, this is the only pro-$p$-cyclic Demu\v skin group.
Indeed, the only infinite pro-$p$-cyclic group is $G=\dbZ_p$, for which $H^2(G)=0$. 
\end{exam}

As shown by Serre, every Demu\v skin pro-$p$ group $G$ is equipped with a canonical homomorphism $\theta\colon G\to\dbZ_p^{\times,1}$ with the property that the cyclotomic pro-$p$ pair $\calG=(G,\theta)$ has the formal Hilbert 90 property \cite{Serre63}*{Section 9.3}.

\begin{defin}
\label{cyclotomic pairs of Demushkin type}
\rm
A cyclotomic pro-$p$ pair $\calG=(G,\theta)$ is of \textsl{Demu\v skin type} if $G$ is a Demu\v skin pro-$p$ group, and $\theta$ is the canonical character associated with $G$ as above.
\end{defin}

For a pro-$p$ Demu\v skin group $G$, the abelianization $G/[G,G]$ has the form $\dbZ_p^{n-1}\times(\dbZ_p/q\dbZ_p)$, where either $q>1$ is a $p$-power, or $q=0$.
Alternatively, $q$ is the maximal $p$-power $>1$ such that $\Img(\theta)\subseteq 1+q\dbZ_p$, and is $0$ if no such maximal $p$-power exists \cite{Labute67}. 
We note that when $q=2$, the index $(\Img(\theta):\Img(\theta)^2)$ is either $2$ or $4$.

By results of Demu\v skin \cite{Demushkin61}, Serre \cite{Serre63}, and Labute \cite{Labute67} (see also \cite{NeukirchSchmidtWingberg}*{\S3.9}), pairs of Demu\v skin type have the following explicit presentation in terms of pro-$p$ generators and relations (which we denote by $\langle\cdots\rangle_{{\rm pro-}p}$):

\medskip

\textsl{Case I:\ }
$q\neq2$. \ \ 
Then the rank $n$ is necessarily even, and 
\[
G\isom\langle x_1\nek x_n\ |\  x_1^q[x_1,x_2]\cdots[x_{n-1},x_n]=1\rangle_{{\rm pro}-p}.
\]
Here $\theta$ is given by $\theta(x_2)=1/(1-q)$ and $\theta(x_i)=1$ for any other $i$. 

\medskip

\textsl{Case II:\ }
$q=2$ and $n$ is odd.\ \ 
Then $p=2$ and
\[
G\isom\langle x_1\nek x_n\ |\ x_1^2x_2^{2^f}[x_2,x_3][x_4,x_5]\cdots[x_{n-1},x_n]=1\rangle_{{\rm pro}-2}
\]
for some $f\in\{2,3\nek \infty\}$ (where by convention $2^\infty=0$).
Here $\theta$ is given by $\theta(x_1)=-1$, $\theta(x_3)=1/(1-2^f)$, and $\theta(x_i)=1$ for any other $i$.

\medskip

\textsl{Case III:\ }
$q=2$, $n$ is even, and $(\Img(\theta):(\Img(\theta)^2)=2$.\ \ 
Then $p=2$ and
\[
G\isom\langle x_1\nek x_n\ |\ x_1^{2+2^f}[x_1,x_2][x_3,x_4][x_5,x_6]\cdots[x_{n-1},x_n]=1\rangle_{{\rm pro}-2}
\] 
for some $f\in\{2,3\nek \infty\}$.
Here $\theta$ is given by $\theta(x_2)=-1/(1+2^f)$ and $\theta(x_i)=1$ for any other $i$.

\medskip

\textsl{Case IV:\ }
$q=2$, $n$ is even, and $(\Img(\theta):(\Img(\theta)^2)=4$.\ \ 
Then $p=2$ and
\[
G\isom\langle x_1\nek x_n\ |\ x_1^2[x_1,x_2]x_3^{2^f}[x_3,x_4][x_5,x_6]\cdots[x_{n-1},x_n]=1\rangle_{{\rm pro}-2}
\] 
for some integer $f\geq2$.
Here $\theta$ is given by $\theta(x_2)=-1$, $\theta(x_4)=1/(1-2^f)$, and $\theta(x_i)=1$ for any other $i$.

\medskip

Conversely, every pro-$p$ group with generators $x_1\nek x_n$ and a single defining relation as in cases I--IV is a Demu\v skin group.

\subsection{Demu\v skin groups as Galois groups}
The main motivating example for Demu\v skin pro-$p$ groups comes from Galois theory of $p$-adic fields:

\begin{exam}
\rm
Consider a cyclotomic pro-$p$ pair of $p$-adic type, that is, a pair $\calG_F(p)$ where $F$ is a finite extension of $\dbQ_p(\mu_p)$ (Definition \ref{pairs of p-adic type}).
Then the Galois group $G=G_F(p)$ is a pro-$p$ Demu\v skin group of rank $n=[F:\dbQ_p]+2$ \cite{NeukirchSchmidtWingberg}*{Th.\ 7.5.11}.
The non-degeneracy of the cup product (property (iii)) reflects the fact the Hilbert symbol in local class field theory is non-degenerate.
Then $\theta$ is the cyclotomic character $\chi_F$ of $F$.
\end{exam}

\begin{exam}
\rm
The Galois group $G_{\dbQ_2}(2)$ is the pro-$2$ group generated by $x_1,x_2,x_3$ subject to a defining relation $x_1^2x_2^4[x_2,x_3]=1$. 
The character $\theta$ is given by $\theta(x_1)=-1$, $\theta(x_2)=1$, and $\theta(x_3)=-1/3=\sum_{k=0}^\infty 2^{2k}$ in $\dbZ_2^\times$.
\end{exam}

Every cyclotomic pair of $p$-adic type has both Galois type and Demu\v skin type, and its rank is $\geq3$.
It is currently an open question whether the opposite implication holds (see \cite{JacobWare89}*{p.\ 395} and \cite{MinacWare92}*{p.\ 339}):

\begin{conj}
\label{conjecture on Demuskin groups as p-adic groups}
If $F$ is a field containing a root of unity of order $p$ and if $\calG_F(p)$ is a Demu\v skin cyclotomic pro-$p$ pair of rank $\geq3$, then $\calG_F(p)$ has $p$-adic type.
\end{conj}

This conjecture is currently open for all cyclotomic pairs of Galois type which are not of $p$-adic type.

\begin{exam}
\rm
Conjecture \ref{conjecture on Demuskin groups as p-adic groups} rules out the appearance of many pro-$p$ Demu\v skin groups as maximal pro-$p$ Galois groups $G_F(p)$, the simplest example being
\[
G=\langle x_1,x_2,x_3\ | x_1^2[x_2,x_3]=1\rangle_{{\rm pro}-2}.
\]
See \cite{JacobWare89}*{Remark 5.5}.
It also rules out the appearance of  \textsl{pro-$p$ surface group} of even rank $n\geq4$ 
\[
G=\langle x_1\nek x_n\ | [x_1,x_2]\cdots[x_{n-1},x_n]=1\rangle_{{\rm pro}-p}.
\]
\end{exam}

\begin{conj}
\label{conjecture on no Demuskin absolute Galois groups}
No pro-$p$ Demu\v skin group of rank $\geq3$ can occur as an \textsl{absolute} Galois group.
\end{conj}

In the next section we shall see that both Conjecture \ref{conjecture on Demuskin groups as p-adic groups} and Conjecture \ref{conjecture on no Demuskin absolute Galois groups} follow from a general conjecture on the arithmetic structure of the fields with $G_F(p)$ a Demu\v skin group of rank $\geq3$.

\begin{rem}
\rm
Labute \cite{Labute66} and Min\'a\v c and Ware (\cite{MinacWare91}, \cite{MinacWare92}) studied a variant of Definition \ref{Demuskin groups}, replacing condition (i) with the assumption that $\rank(G)=\aleph_0$. 
The similar notion when $\rank(G)>\aleph_0$ was studied by Bar-On and Nikolov \cite{BarOnNikolov24}.

It turns out that the above conjectures become false for these variants.
For instance, Min\'a\v c and Ware use transcendental methods to construct fields $F$ for which $G_F(p)$ has rank $\aleph_0$ and (ii) and (iii) hold, but which do not arise in any natural way from a $p$-adic field, in contrast to Conjecture \ref{conjecture on Demuskin groups as p-adic groups}.
In fact, in their construction $F$ may even have a positive characteristic.

Similarly, the $p$-Sylow subgroup of $G_{\dbQ_p}$ has rank $\aleph_0$ and satisfies (ii) and (iii) \cite{Labute66}*{Th.\ 5}, in contrast to Conjecture \ref{conjecture on no Demuskin absolute Galois groups}.
\end{rem}

\subsection{Indecomposability of Demu\v skin groups}
Next, we show that cyclotomic pro-$p$ pairs of Demu\v skin type and of rank $\geq3$ are `atoms' in the category of cyclotomic pro-$p$ pairs, in the sense that they cannot be decomposed using the two operations of free product and extension.

\begin{lem}
\label{Demuskin groups are indecomposable}
A pro-$p$ Demu\v skin group $G$ cannot be decomposed as a free pro-$p$ product $G=G_1*_pG_2$ of nontrivial closed subgroups $G_1,G_2$ of $G$. 
\end{lem}
\begin{proof}
Assume the contrary.
Then 
\[
H^2(G)=H^2(G_1)\oplus H^2(G_2)
\]
via the restriction maps.
By (ii), this cohomology group is $1$-dimensional over $\dbF_p$, so without loss of generality, $H^2(G_2)=0$.
Then $G_2$ is a free pro-$p$ group \cite{NeukirchSchmidtWingberg}*{Prop.\ 3.5.17}.
Since it is nontrivial and finitely generated (as a quotient of $G$), it is a free pro-$p$ product of a nonempty set of copies of $\dbZ_p$.
We may therefore assume that $G_2=\dbZ_p$.

Let $\psi\colon G\to\dbZ/p$ be the unique continuous homomorphism which is trivial on $G_1$ and is the projection map $\dbZ_p\to\dbZ/p$ on $G_2$.
We consider it as a nonzero element of $H^1(G)$.
Its restriction to $H^1(G_1)$ is $0$.
For every $\varphi\in H^1(G)$ the restriction of $\varphi\cup\psi$ is $0$ in both $H^2(G_1)$ and $H^2(G_2)=0$.
Therefore $\varphi\cup\psi=0$ in $H^2(G)$, contradicting (iii).
\end{proof}

\begin{lem}
\label{Demuskin groups are not extensions}
A cyclotomic pro-$p$ pair $\calG$ of Demu\v skin type and of rank $\geq3$ cannot be an extension $\dbZ_p\rtimes\bar\calG$.
\end{lem}
\begin{proof}
Suppose that $\calG=A\rtimes\bar \calG$ has Demu\v skin type, where $A\isom\dbZ_p$.
Write $\calG=(G,\theta)$ and $\bar\calG=(\bar G,\bar\theta)$.
By Proposition \ref{generalized Wadsworth formula}, $H^2(G)\isom H^2(\bar G)\oplus H^1(\bar G)$.

Note that $H^1(\bar G)$ cannot be trivial, since then $\bar G$ would be trivial, so $G=A\isom\dbZ_p$, which is not Demu\v skin.
As $\dim_{\dbF_p}H^2(G)=1$, this implies that $H^2(\bar G)=0$ and $\dim_{\dbF_p}H^1(\bar G)=1$, meaning that $\bar G$ is a free pro-$p$ group of rank $1$, i.e., $\bar G\isom\dbZ_p$.
It follows that $G=\dbZ_p\rtimes\dbZ_p$, which has rank $2$.
\end{proof}

\section{Arithmetically Demu\v skin fields}
\label{section on arithmetically Demushkin fields}
In this section we explain a conjecture made in \cite{Efrat03}, according to which all pro-$p$ Demu\v skin groups of rank $\geq3$ which are realizable as maximal pro-$p$ Galois groups of fields arise from certain valuations which generalize the $p$-adic valuations on $\dbQ_p$.

\subsection{Milnor $K$-theory}
Let $n\geq0$.
Recall that the \textsl{$n$th Milnor $K$-group} of a field $F$ is 
\[
K^M_n(F)=(F^\times)^{\tensor n}/\mathrm{St}_n
\]
where $(F^\times)^{\tensor n}$ is the $n$th tensor power of $F^\times$ and $\mathrm{St}_n$ is its ideal generated by all pure tensors $a_1\tensor\cdots \tensor a_n$ such that $a_i+a_j=1$ for some $i<j$ (the \textsl{Steinberg relations}).
The natural map $\{\cdot,\cdot\}\colon F^\times\tensor_\dbZ F^\times\to K^M_2(F)$ induces a (functorial) augmented $\dbF_p$-bilinear map
\[
\overline{\{\cdot,\cdot\}}\colon F^\times/(F^\times)^p\times F^\times/(F^\times)^p\to K^M_2(F)/pK^M_2(F).
\]

When $F$ contains a root of unity of order $p$, the Kummer isomorphism gives a commutative diagram of bilinear maps
\begin{equation}
\label{Mer-Sus}
\xymatrix{
F^\times/(F^\times)^p\ar[d]^{\wr}&*-<3pc>{\times}& F^\times/(F^\times)^p\ar[d]^{\wr}\ar[r]^{\overline{\{\cdot,\cdot\}}\ \quad}& K^M_2(F)/pK^M_2(F)\ar[d]\\
H^1(G_F(p))&*-<3pc>{\times}&H^1(G_F(p))\ar[r]^{\cup}&H^2(G_F(p)),
}
\end{equation}
where the right vertical map is the \textsl{Galois symbol} of degree $2$;
By the Merkurjev--Suslin theorem, it is also an isomorphism \cite{GilleSzamuely17}.

\subsection{Definition and Galois structure}
\begin{defin}
\label{def of arithmetically Demuskin fields}
\rm
We say that a field $F$ containing a root of unity of order $p$ is \textsl{arithmetically Demu\v skin} if there exists a valuation $u$ on $F$ and a valuation $w$ on the residue field $E=\bar F_u$ with the following properties:
\begin{enumerate}
\item[(i)]
The value group of $u$ satisfies $\Gam_u=p\Gam_u$;
\item[(ii)]
$\Char\,E=0$ (whence also $\Char\,F=0$);
\item[(iii)]
The valuation $w$ is discrete;
\item[(iv)]
$\bar E_w$ is a finite field of characteristic $p$;
\item[(v)]
$1+\grm_u\leq(F^\times)^p$;
\item[(vi)]
$1+p^2\grm_w\leq (E^\times)^p$.
\end{enumerate}
\end{defin}
Conditions (v) and (vi) are weaker versions of Hensel's lemma; See \cite{Efrat06}*{Section 18.2}.

We also take $v$ to be the refinement of $u$ such that $w=v/u$ (see \S\ref{section on valuations}).

\begin{prop}
\label{properties of the completion}
Let $F$ be an arithmetically Demu\v skin field and let $E$ and $w$ be as above.
Let $(\hat E,\hat w)$ be the completion of the discretely valued field $(E,w)$.
Then:
\begin{enumerate}
\item[(a)]
The valued field $(\hat E,\hat w)$ is a finite extension of $\dbQ_p(\mu_q)$ with its $p$-adic valuation;
\item[(b)]
$\calG_{\hat E}(p)\isom \calG_E(p)$ via the restriction homomorphism;
\item[(c)]
The valued fields $(E,w)$, $(F,u)$ and $(F,v)$ are $p$-Henselian.
\end{enumerate}
\end{prop}
\begin{proof}
(a) \quad
See \cite{Serre67}*{Ch.\ II, Sect.\ 5}.

\medskip

(b) \quad
The functoriality of the pairing $\overline{\{\cdot,\cdot\}}$ gives a commutative diagram
\[
\xymatrix{
E^\times/(E^\times)^p\ar[d]&*-<3pc>{\times}& E^\times/(E^\times)^p\ar[d]\ar[r]& K^M_2(E)/pK^M_2(E)\ar[d] \\
\hat E^\times/(\hat E^\times)^p&*-<3pc>{\times}& \hat E^\times/(\hat E^\times)^p\ar[r]&K^M_2(\hat E)/pK^M_2(\hat E).
}
\]
By (vi), $(E^\times)^p$ is open in the $w$-topology on $E$.
Hence, by \cite{Efrat03}*{Prop.\ 4.3(c)}, the vertical maps in the diagram are isomorphisms.

By (\ref{Mer-Sus}), this means that the restriction maps $H^l(G_E(p))\to H^l(G_{\hat E}(p))$, $l=1,2$, are isomorphisms compatible with the cup product.
By \cite{Serre65}*{Lemma 2}, the restriction map $G_{\hat E}(p)\to G_E(p)$ is an isomorphism.
It is clearly compatible with the cyclotomic characters.

\medskip

(c) \quad
By (b), $E=E(p)\cap \hat E$.
Since $(\hat E,\hat w)$ is ($p$-)Henselian, this implies that $(E,w)$ is $p$-Henselian \cite{Efrat06}*{Prop.\ 15.3.3}.

The $p$-Henselity of $(F,u)$ follows from (ii) and (iv).
The $p$-Henselity of $(F,v)$ follows from that of $(E,w)$ and $(F,u)$ \cite{Efrat06}*{Th.\ 20.3.2}. 
\end{proof}

We recover \cite{Efrat03}*{Th.\ 6.3(a)}:

\begin{cor}
\label{arithmetically Demuskin implies Demuskin}
If $F$ is an arithmetically Demu\v skin field containing a root of unity of order $p$, then $\calG_F(p)$ is a cyclotomic pro-$p$ pair of $p$-adic type.
\end{cor}
\begin{proof}
Since $(F,u)$ is $p$-Henselian (by Proposition \ref{properties of the completion}(c)), $\Gam_u=p\Gam_u$ and $\Char\,E=0$ (by (i) and (ii), respectively), Proposition \ref{extensions from tame valuations} shows that $\calG_F(p)\isom\calG_E(p)$. 
By Proposition \ref{properties of the completion}, $\calG_E(p)\isom\calG_{\hat E}(p)$ has $p$-adic type.
\end{proof}

Our main conjecture on Demu\v skin maximal pro-$p$ Galois groups is that the converse of Corollary \ref{arithmetically Demuskin implies Demuskin} also holds: 

\begin{conj}[Demu\v skin Conjecture]
\label{strong conjecture on Demuskin fields}
If $F$ is a field containing a root of unity of order $p$ and such that $G_F(p)$ is a pro-$p$ Demu\v skin group of rank $\geq3$, then $F$ is arithmetically Demu\v skin.
\end{conj}

In the special case $p=2$ and $G_F(2)\isom G_{\dbQ_2}(2)$, positive results on this conjecture are given by Koenigsmann and Strommen \cite{KoenigsmannStrommen24}.

\subsection{Connections between the conjectures}

\begin{thm}
\label{ETC implies Demuskin conjecture}
The arithmetical form of the Elementary Type Conjecture (Conjecture \ref{ETC arithmetical form}) implies the Demu\v skin conjecture (Conjecture \ref{strong conjecture on Demuskin fields}).
\end{thm}
\begin{proof}
Suppose that $F$ is a field containing a root of unity of order $p$ and such that $G_F(p)$ is a Demu\v skin group of rank $\geq3$.
We need to show that $F$ is an arithmetically Demu\v skin field.

By Lemma \ref{Demuskin groups are indecomposable}, $G_F(p)$ is indecomposable as a free pro-$p$ product.
Therefore Conjecture \ref{ETC arithmetical form} implies that $G_F(p)$ has of one of the forms (i)--(iii) listed in the conjecture. 
Being of rank $\geq3$, the group $G_F(p)$ cannot be isomorphic to $\dbZ_p$ nor to $\dbZ/2$, so (i) and (iii) are impossible.
Therefore (ii) holds, that is, $G_F(p)=Z(\tilde v/v)$ is the decomposition group of some extension $\tilde v$ of a valuation $v$ on $F$ to $F(p)$ with a nontrivial inertia group.
Thus $v$ is $p$-Henselian.

If $\Char\,\bar F_v\neq p$, then by Proposition \ref{extensions from tame valuations}, $\calG_F(p)\isom\dbZ_p^m\rtimes\calG_{\bar F_v}(p)$, where $m=\dim_{\dbF_p}(\Gam_v/p\Gam_v)\geq1$.
By (\ref{associativity of extension}), $\calG_F(p)\isom\dbZ_p\rtimes(\dbZ_p^{m-1}\rtimes \calG_{\bar F_v}(p))$, contrary to Lemma \ref{Demuskin groups are not extensions}.

Consequently, $\Char\,\bar F_v=p$, so we are in the situation of Theorem \ref{the wild case} and Corollary \ref{cor to wild case}.
In particular, $\calG_F(p)\isom\dbZ_p^{m'}\rtimes\bar\calG$, with $0\leq m'<\infty$ and $\bar\calG$ is either of $p$-adic type or its underlying group is a finitely generated free pro-$p$ group.
Again, Lemma \ref{Demuskin groups are not extensions} implies that $m'=0$, and hence the underlying group of $\bar\calG$ is the Demu\v skin group $G_F(p)$, which is not free pro-$p$.
Therefore we are in the case where the coarsening $u$ of $v$ and the valuation $w=v/u$ on $E=\bar F_u$ (as in Theorem \ref{the wild case}) satisfy $\Gam_u=p\Gam_u$, and $w$ is discrete with a finite residue field of characteristic $p$.
Moreover, since $v$ is $p$-Henselian, so are $u$ and $w$ (see \S\ref{section on valuations}).
Hence $1+\grm_u\leq (F^\times)^p$ and $1+p^2\grm_w\leq(E^\times)^p$ \cite{Efrat06}*{\S18.2}.
Thus $F$ is an arithmetically Demu\v skin field.
\end{proof}

In fact, the `Lego form' of the Elementary Type Conjecture (Conjecture \ref{ETC Lego version}) together with the Demu\v skin Conjecture (Conjecture \ref{strong conjecture on Demuskin fields}) also imply the arithmetical form of the Elementary Type Conjecture (Conjecture \ref{ETC arithmetical form}).
The proof of this fact is quite subtle, and is developed in \S\ref{section on rigidity}--\S\ref{section on equivalence of conjectures}.

The Demu\v skin Conjecture \ref{strong conjecture on Demuskin fields} implies the conjectures of \S\ref{section on Demushkin groups}:

\begin{thm}
Assume the Demu\v skin Conjecture \ref{strong conjecture on Demuskin fields}.
Then:
\begin{enumerate}
\item[(a)]
Every cyclotomic pro-$p$ pair of rank $\geq3$ which is both of Galois type and of Demu\v skin type, is of $p$-adic type.
\item[(b)]
No pro-$p$ Demu\v skin group of rank $\geq3$ can occur as an absolute Galois group.
\end{enumerate}
\end{thm}
\begin{proof}
(a) \quad
This follows from Corollary \ref{arithmetically Demuskin implies Demuskin}.

\medskip

(b) \quad
Suppose that $F$ is a field such that $G_F$ is a pro-$p$ Demu\v skin group of rank $\geq3$.
In particular, $F$ contains a root of unity of order $p$.
By Conjecture \ref{strong conjecture on Demuskin fields}, $F$ is arithmetically Demu\v skin.
Let $u$ and $E$ be as in Definition \ref{def of arithmetically Demuskin fields}.
By Proposition \ref{properties of the completion}(c), $(F,u)$ is $p$-Henselian.
As $G_F$ is pro-$p$, $(F,u)$ is in fact Henselian.
By (\ref{exact sequence for decomposition group}), $G_E=G_{\bar F_u}$ is a quotient of $G_F$.
But by Lemma \ref{not pro p}, $G_E$ is not pro-$p$, a contradiction.
\end{proof}

\begin{rem} 
\label{existence of a valuation v on a Demuskin field}
\rm
In \cite{Efrat03}*{Th.\ 6.3(b)} Conjecture \ref{strong conjecture on Demuskin fields} is proved under the additional assumption that $F$ is equipped with a valuation $v$ such that $\Gam_v\neq p\Gam_v$ and the decomposition field of $v$ in $G_F(p)$ does not contain $\mu_{p^\infty}$. 

The opposite implication holds unconditionally, namely, for every arithmetically $p$-Demu\v skin field $F$ the valuation $v$ defined above satisfies these two properties.
Indeed, its decomposition group is all of $G_F(p)$, by Proposition \ref{properties of the completion}(c). 
By (\ref{exact sequence for coarsenings}) and (i), $\Gam_v/p\Gam_v\isom\Gam_w/p\Gam_w\isom\dbZ/p$.
Further, as in the proof of Corollary \ref{arithmetically Demuskin implies Demuskin}, the residue and restriction maps, respectively, give isomorphisms $\calG_F(p)\isom\calG_E(p)\isom\calG_{\hat E}(p)$.
Since the pro-$p$ cyclotomic character of $G_{\hat E}(p)$ is nontrivial, so is that of $G_F(p)$, i.e., $\mu_{p^\infty}\not\subseteq F$.
\end{rem}

\subsection{Trichotomic elements}
\label{subsection on trichotomic elements}
The following Milnor $K$-theoretic criterion for the existence of a valuation $v$ with $\Gam_v\neq p\Gam_v$, as in Remark \ref{existence of a valuation v on a Demuskin field}, is proved in \cite{Efrat03}*{Cor.\ 8.3}:

\begin{thm}
Suppose that $p\neq2$ and let $F$ be a field containing a root of unity of order $p$ with $G_F(p)$ a pro-$p$ Demu\v skin group of rank $\geq3$. 
The following conditions are equivalent:
\begin{enumerate}
\item[(a)]
There exists a valuation $v$ on $F$ with $\Gam_v\neq p\Gam_v$;
\item[(b)]
There exists $a\in F^\times\setminus(F^\times)^p$ such that for every $b\in F$, $b\neq0,1$, at least one of $\overline{\{a,b\}}$, $\overline{\{a,1-b\}}$, and $\overline{\{a,1-b\inv\}}$ is non-zero in $K^M_2(F)/pK^M_2(F)$.
\end{enumerate}
\end{thm}
Elements $a$ as in (b) are called in \cite{Efrat03} \textsl{trichotomic}.
When $F$ is arithmetically $p$-Demu\v skin, with $u$, $w$, $v$ as before, these are the elements $a\in F^\times$ such that $F(\root p\of a)/F$ is a nontrivial unramified extension of $F$ at $v$ \cite{Efrat03}*{Th.\ 8.3}.
More explicitly, these are the elements of the form 
\[
a=1-\omega(1-\zeta)pc^p,
\]
where $1\neq\zeta\in\mu_p$, $c\in F^\times$, and  $\omega$ is a $v$-integral element of $F$ with $\Tr(\bar\omega)\neq0$ for the trace map $\Tr\colon\bar E_w\to\dbF_p$ \cite{Koch92}*{Ch.\ I, Prop.\ 1.85}.

\section{Rigidity}
\label{section on rigidity}
Our next aim is to prove a converse to Corollary \ref{cor to ETC} and Theorem \ref{ETC implies Demuskin conjecture}, showing that the Lego form of the Elementary Type Conjecture, together with the Demu\v skin conjecture, implies its arithmetical form.
This will be attained in \S\ref{section on equivalence of conjectures}, after some preparations. 
Crucial to this result is the notion of \textsl{rigidity}, which we now describe.

\subsection{The general formalism}
We define an \textsl{augmented $\dbF_p$-bilinear map} to be a  bilinear map  $(\cdot,\cdot)\colon A_1\times A_1\to A_2$ of $\dbF_p$-linear spaces, where $A_1$ has a distinguished element $\eps_A$ satisfying $2\eps_A=0$ (so $\eps_A=0$ for $p\neq2$).
A \textsl{morphism} between augmented $\dbF_p$-bilinear maps $(\cdot,\cdot)$, $(\cdot,\cdot)'$ consists of group homomorphisms $\varphi_1,\varphi_2$ with a commutative diagram
\begin{equation}
\label{morphism of augmented bilinear maps}
\xymatrix{
A_1\ar[d]_{\varphi_1}&*-<3pc>{\times}&A_1\ar[r]^{(\cdot,\cdot)}\ar[d]_{\varphi_1}&A_2\ar[d]^{\varphi_2}\\
A'_1&*-<3pc>{\times}&A'_1\ar[r]^{(\cdot,\cdot)'}&A'_2,
}
\end{equation}
and such that $\varphi_1(\eps_A)=\eps_{A'}$ (with the obvious notation). 

\begin{defin}
\rm
Given an augmented $\dbF_p$-bilinear map as above, we call $0\neq a\in A_1$ \textsl{rigid} if for every $b\in A_1$ with $(a,b)=0\in A_2$, the elements $\eps_A+a$ and $b$ are $\dbF_p$-linearly dependent in $A_1$.
\end{defin}

The proof of the following lemma is straightforward:  

\begin{lem}
\label{bilinear pairings and rigidity}
Consider a morphism of augmented $\dbF_p$-bilinear maps as in (\ref{morphism of augmented bilinear maps}) with $\varphi_1$ injective.
If $0\neq a\in A_1$ and if $\varphi_1(a)$ is rigid with respect to $(\cdot,\cdot)'$, then $a$ is rigid with respect to $(\cdot,\cdot)$.
\end{lem}

\begin{exam}
\rm
(1) \quad
Given a cyclotomic pro-$p$ pair $\calG=(G,\theta)$, the \textsl{augmented $\dbF_p$-bilinear map associated with $\calG$} is the cup product map
\[
\cup\colon H^1(G)\times H^1(G)\to H^2(G),
\]
with the distinguished element $\eps=\eps_\calG$ of (\ref{eps calG}).

\smallskip

(2) \quad
For a field $F$, the natural map (\S\ref{subsection on trichotomic elements})
\[
\overline{\{\cdot,\cdot\}}\colon F^\times/(F^\times)^p\times F^\times/(F^\times)^p\to K^M_2(F)/pK^M_2(F)
\]
is an augmented $\dbF_p$-bilinear map with the distinguished element $\eps=\overline{-1}$.
\end{exam}

\begin{rem}
\label{MerSus 2}
\rm
When $F$ contains a root of unity of order $p$, these two maps may be identified via the commutative diagram (\ref{Mer-Sus}), which is an isomorphism of augmented $\dbF_p$-bilinear maps.
Indeed, we have $\eps_{\calG_F(p)}=(-1)_F$ (Remark \ref{aaa}(2)).
\end{rem}

\subsection{A cohomological identity}
\label{subsection on cohomological identity}
We now assume again that $\calG= A\rtimes\bar\calG$ is an extension as in \S\ref{section on cohomology of extension}.
We show that this group structure naturally gives rise to rigid elements.
To this end we first prove a group-theoretic identity, which is a partial analog of the identity $\{-a,a\}=0$ in Milnor $K$-theory \cite{Efrat06}*{Prop.\ 24.1.1}.
We write $\calG=(G,\theta)$ and $\bar\calG=(\bar G,\bar\theta)$.

\begin{prop}
\label{dihedral identity}
For every $\beta\in H^1(G)$ with $\Res_A\beta\neq0$ in $H^1(A)$ and $\Res_{\bar G}\beta=0$ in $H^1(\bar G)$, one has $\beta\cup(\eps_\calG+\beta)=0$ in $H^2(G)$ (with $\eps_\calG$ as in (\ref{eps calG})).
\end{prop}
\begin{proof}
We distinguish between 3 cases:

\case  $p>2$. \rm
Here $\eps_\calG=0$, so the assertion follows from the anti-commutativity of the cup product.

\case  $p=2$ and $\eps_\calG=0$. \rm
Here $\Img(\bar\theta)=\Img(\theta)\subseteq1+4\dbZ_2$.
The restriction $\Res_A\beta\colon A\to\dbZ/2$ factors through an epimorphism $\lam\colon A\to\dbZ/4$.
We obtain a homomorphism
\[
\lam\rtimes1\colon G=A\rtimes\bar G\to\dbZ/4.
\]
There is a commutative diagram with an exact row
\begin{equation}
\label{first extension}
\xymatrix{
&&&G\ar[ld]_{\lam\rtimes1}\ar[d]^{\beta}&\\
0\ar[r]&\dbZ/2\ar[r]&\dbZ/4\ar[r]&\dbZ/2\ar[r]&0.
}
\end{equation}
Write $\tilde G=G/\Ker(\beta)\isom\dbZ/2$ and let $\tilde\beta\colon \tilde G\to\dbZ/2$ be the isomorphism induced by $\beta$.
By \cite{EfratMinac11}*{Prop.\ 9.1(c)},  $\tilde\beta\cup\tilde\beta\in H^2(\tilde G)$ corresponds to the group extension in (\ref{first extension}).
Inflating to $G$, we obtain from Hoechsmann lemma that $\beta\cup\beta=0$ in $H^2(G)$ \cite{NeukirchSchmidtWingberg}*{Prop.\ 3.5.9}.

\case  $p=2$ and $\eps_\calG\neq0$. \rm
As before, $\Res_A\beta\colon A\to\dbZ/2$ factors through an epimorphism $\lam\colon A\to\dbZ/4$.
Consider the dihedral group of order $8$ 
\[
D_4=(\dbZ/4)\rtimes(\dbZ/2)=\langle r,s\ | \ r^4=s^2=(rs)^2=1\rangle.
\]
We note the following epimorphisms:
\[
\lam\rtimes\eps_{\bar\calG}\colon G=A\rtimes\bar G\to D_4,\quad \Img(\lam)=\langle r\rangle,\,\, \Img(\eps_{\bar\calG})=\langle s\rangle,
\]
\[
\tilde\beta\colon D_4\to\dbZ/2, \quad r\mapsto \bar1, \ \ s\mapsto \bar0,
\]
\[
\tilde\eps\colon D_4\to\dbZ/2, \quad r\mapsto\bar0,\ \  s\mapsto\bar1.
\]
Note that
\[
\tilde\beta\circ(\lam\rtimes\eps_{\bar\calG})=\beta, \qquad  
 \tilde\eps\circ(\lam\rtimes\eps_{\bar\calG})=\eps_\calG.
\]
We obtain an epimorphism
\[
\rho=(\tilde\beta,\tilde\eps+\tilde\beta)\colon D_4\to (\dbZ/2)^2, \quad r\mapsto(\bar1,\bar1),\ \  s\mapsto(\bar0,\bar1).
\]
with kernel $\langle r^2\rangle\isom\dbZ/2$.
There is a commutative diagram with an exact row
\begin{equation}
\label{second extension}
\xymatrix{
&&&G\ar[ld]_{\lam\rtimes\eps_{\calG}}\ar[d]^{(\beta,\eps_\calG+\beta)}&\\
0\ar[r]&\dbZ/2\ar[r]& D_4\ar[r]_{\rho}&(\dbZ/2)^2\ar[r]&0.
}
\end{equation}
Since $\tilde\beta\neq\tilde\eps+\tilde\beta$, we may identify $\tilde G:=G/(\Ker(\tilde\beta)\cap\Ker(\tilde\eps+\tilde\beta))\isom(\dbZ/2)^2$.
Let $\bar\beta,\bar\eps\colon\tilde G\to\dbZ/2$ be the homomorphisms induced by $\tilde\beta,\tilde\eps$, respectively.
By \cite{EfratMinac11}*{Prop.\ 9.1(e)}, $\bar\beta\cup (\bar\eps+\bar\beta)\in H^2(\tilde G)$ corresponds to the group extension in (\ref{second extension}).
Inflating to $G$, we obtain as before from Hoechsmann lemma that $\beta\cup(\eps_\calG+\beta)=0$ in $H^2(G)$.
\end{proof}

\subsection{Rigidity in extensions}
As in \S\ref{section on cohomology of extension}, and with the same setup as in \S\ref{subsection on cohomological identity}, we choose elements $\beta_l\in H^1(G)$, $l\in L$, whose restrictions to $H^1(A)$ form an $\dbF_p$-linear basis and such that $\Res_{\bar G}\beta_l=0$.

\begin{prop}
\label{criterion for rigidity}
Every $\varphi\in H^1(G)\setminus\inf(H^1(\bar G))$ is rigid with respect to the cup product augmented bilinear map.
\end{prop}
\begin{proof}
We compute using Proposition \ref{generalized Wadsworth formula}.
For simplicity we identify $H^1(\bar G)$ with its image in $H^1(G)$ under inflation.

Suppose that $\varphi\cup\psi=0$ for some $\psi\in H^1(G)$.
Write 
\[
\varphi=\bar\varphi+\sum_{l\in L}j_l\beta_l, \quad \psi=\bar\psi+\sum_{l'\in L}k_{l'}\beta_{l'}
\]
with $\bar\varphi,\bar\psi\in H^1(\bar G)$ and $0\leq j_l, k_{l'}\leq p-1$ which are almost always $0$.
Then
\[
\bar\varphi\cup\bar\psi+\sum_lj_l\beta_l\cup\bar\psi+\sum_{l'}k_{l'}\bar\varphi\cup\beta_{l'}
+\sum_{l,l'}j_lk_{l'}\beta_l\cup\beta_{l'}=0.
\]
By Proposition \ref{dihedral identity}, $\beta_l\cup\beta_l=\eps_\calG\cup\beta_l$ for every $l$.
This and the anti-commutativity of the cup product imply that
\[
\bar\varphi\cup\bar\psi+\sum_l(-j_l\bar\psi+k_l\bar\varphi+j_lk_l\eps_\calG)\cup\beta_l
+\sum_{l<l'}(j_lk_{l'}-j_{l'}k_l)\beta_l\cup\beta_{l'}=0.
\]
From the direct sum decomposition of $H^2(G)$ as in Proposition \ref{generalized Wadsworth formula} we get
\[
\begin{split}
&j_l\bar\psi=k_l(\bar\varphi+j_l\eps_\calG) \hbox{ for every }l\in L,\\
&j_lk_{l'}=j_{l'}k_l\hbox{ for every }l,l'\in L.
\end{split}
\]
It follows that for every $l$,
\[
j_l\psi=j_l(\bar\psi+\sum_{l'}k_{l'}\beta_{l'})=k_l(\bar\varphi+\sum_{l'}j_{l'}\beta_{l'}+j_l\eps_\calG)
=k_l(\varphi+j_l\eps_\calG).
\]
Now the assumption that $\varphi\not\in H^1(\bar G)$ means that $j_l\neq0$ for some $l$.
Also note that $\eps_\calG=0$ unless $p=2$, and in this case $j_l=1$.
We deduce that $\psi\in\langle\varphi+\eps_\calG\rangle$, as desired.
\end{proof}

\section{Construction of valuations and rigidity}
\label{section on construction of valuations and rigidity}
Let $F$ be again a field containing a root of unity of order $p$.
We prove a partial converse of Proposition \ref{extensions from tame valuations}, namely, that if $\calG_F(p)\isom\dbZ_p^m\rtimes\bar\calG$ with $m\geq1$, then, with certain well-understood exceptions, $F$ is equipped with a $p$-Henselian valuation $v$ such that $(\Gam_v:p\Gam_v)\geq m$ and $\Char\,\bar F_v\neq p$.

For a subgroup $S$ of $F^\times$ we define the \textsl{Milnor $K$-group of $F$ modulo $S$} to be 
\[
K^M_n(F)/S=(F^\times/S)^{\tensor n}/\mathrm{St}_n(S),
\]
where $\mathrm{St}_n(S)$ is the subgroup generated by all pure tensors $a_1S\tensor\cdots\tensor a_nS$ such that $1\in a_iS+a_jS$ for some $i<j$ \cite{Efrat06}*{\S24.1}.
For $S=\{1\}$ this is just $K^M_n(F)$, and for $S=(F^\times)^p$ this recovers $K^M_n(F)/pK^M_n(F)$.

Every pure tensor $a_1S\tensor\cdots\tensor a_nS$ such that $a_iS=-a_jS$ for some $i<j$ belongs to $\mathrm{St}_n(S)$.
Indeed, this follows from
\[
a_jS=(1-a_i)(1-a_i\inv)\inv S.
\]
The following conditions are equivalent:
\begin{enumerate}
\item[(i)]
For every $n\geq2$ the group ${\rm St}_n(S)$ is generated by the pure tensors $a_1S\tensor\cdots\tensor a_nS$ such that $a_iS=-a_jS$ for some $i<j$
\item[(ii)]
The group ${\rm St}_2(S)$ is generated by all pure tensors $aS\tensor (-aS)$, $a\in F^\times$.
\end{enumerate}
In this case we say that the subgroup $S$ is \textsl{totally rigid} \cite{Efrat06}*{p.\ 235}.

We now take $S=(F^\times)^p$.
Let $N$ be the subgroup of $F^\times$ generated by $-1$, $(F^\times)^p$, and by all non-rigid elements in $F^\times$ with respect to $\overline{\{\cdot,\cdot\}}$.
When $S=(F^\times)^p$ is totally rigid, $N=\langle-1,(F^\times)^p\rangle$ \cite{Efrat06}*{Cor.\ 26.4.7}.

We will need the following special cases of \cite{Efrat06}*{Th.\ 26.5.5} and \cite{Efrat06}*{Th.\ 26.6.1}, respectively:

\begin{thm}
\label{Efr0626 5 5}
\begin{enumerate}
\item[(a)]
Suppose that $(F^\times)^p$ is not totally rigid and that $(F^\times:(F^\times)^p)<\infty$ or $p=2$.
Then there exists a valuation $v$ on $F$ such that $1+\grm_v\leq (F^\times)^p$ and $N=(F^\times)^pO_v^\times$.
\item[(b)]
Suppose that $(F^\times)^p$ is totally rigid.
Then there exists a valuation $v$ on $F$ such that $1+\grm_v\leq (F^\times)^p$ and $N$ is a subgroup of $(F^\times)^pO_v^\times$ of index dividing $p$.
\end{enumerate}
\end{thm}

We also record the following special case, where we use the results and terminology of \cite{ArasonElmanJacob87} on the valuation $O(S,H)$ discussed in \S\ref{section on valuations}.

\begin{lem}
\label{AEJ lemma}
Let $p=2$ and let $F$ be a non-Pythagorean field of characteristic $\neq2$.
Suppose that $S=(F^\times)^2$ is totally rigid,  $(F^\times:S)=4$, and $-1\in S$.
Then there is a valuation $v$ on $F$ with $1+\grm_v\leq S$ and $(SO_v^\times:S)=2$.
\end{lem}
\begin{proof}
Our assumptions imply that $F$ is \textsl{not exceptional} in the sense of \cite{ArasonElmanJacob87}*{Def.\ 2.15}.
Let $A$ be the set of all $x\in F^\times$ such that $1-x\not\in S\cup -xS$.
In the terminology of \cite{ArasonElmanJacob87}, $B(S)=A\cup-A$ is the set of all \textsl{$S$-basic} elements in $F$.
We note that $AS$ consists of all $x\in F^\times$ such that $S-xS\not\subseteq S\cup -xS$.
By \cite{Efrat06}*{Lemma 26.4.6 and Cor.\ 26.4.7}, $A\subseteq\langle -1,S\rangle$, so $H:=S\langle B(S)\rangle\leq \langle -1,S\rangle$.
Now \cite{ArasonElmanJacob87}*{Th.\ 2.16} says that $O(S,H)$ is a valuation ring with maximal ideal $\grm$ satisfying $1+\grm\leq S$.
Let $v$ be the corresponding valuation on $F$.
Then $O_v^\times\leq H$.
It follows that $(SO_v^\times:S)\leq (H:S)\leq2$, and since $-1\in O_v^\times\setminus S$, these are equalities.
\end{proof}

Now suppose that
\[
\calG=\calG_F(p)\isom\dbZ_p^m\rtimes\bar\calG, \quad m\geq1.
\]
We write as before $\calG=(G,\theta)$ and $\bar\calG=(\bar G,\bar\theta)$. 
By Lemma \ref{cohomological properties of extensions}(a), $H^1(G)=(\dbZ/p)^m\oplus \inf(H^1(\bar G))$.
Let $T\leq F^\times$ correspond under the Kummer isomorphism to the summand $\inf(H^1(\bar G))$.
Thus $(F^\times)^p\leq T$, $-1\in T$, and
\begin{equation}
\label{formula for m}
\dim_{\dbF_p}(F^\times/T)=m.
\end{equation}

\begin{cor}
\label{corollary on rigidity}
$N\leq T$.
\end{cor}
\begin{proof}
By combining Proposition \ref{criterion for rigidity} with Lemma \ref{bilinear pairings and rigidity} and Remark \ref{MerSus 2}, we deduce that every $a\in F^\times\setminus T$ is rigid with respect to $\overline{\{\cdot,\cdot\}}$.
\end{proof}

\begin{thm}
\label{existence of valuations}
\begin{enumerate}
\item[(a)]
Assume that $(F^\times)^p$ is not totally rigid, and further, that $G_F(p)$ is a finitely generated pro-$p$ group or $p=2$.
Then there exists a $p$-Henselian valuation $v$ on $F$ such that  $\Char\,\bar F_v\neq p$ and $\dim_{\dbF_p}(\Gam_v/p\Gam_v)\geq m$.
\item[(b)]
Assume that $(F^\times)^p$ is totally rigid and $-1\in(F^\times)^p$ (resp., $p=2$ and $-1\not\in(F^\times)^2$).
Then there exists a $p$-Henselian valuation $v$ on $F$ such that  $\Char\,\bar F_v\neq p$ and the natural map $F^\times/(F^\times)^p\to\Gam_v/p\Gam_v$ has kernel of order dividing $p$ (resp., $4$).
\item[(c)]
Assume that $p=2$, $(F^\times:(F^\times)^2)=4$, $(F^\times)^2$ is totally rigid, $F$ is non-Pythagorean, and $-1\not\in(F^\times)^2$.
Then there exists a $2$-Henselian valuation $v$ on $F$ such that  $\Char\,\bar F_v\neq 2$ and $\Gam_v\neq 2\Gam_v$.
\end{enumerate}
\end{thm}
\begin{proof}
(a) \quad
When $G_F(p)$ is finitely generated, $(F^\times:(F^\times)^p)<\infty$.
Since in addition $(F^\times)^p$ is not totally rigid, Theorem \ref{Efr0626 5 5}(a) gives rise to a valuation $v$ on $F$ such that 
\[
1+\grm_v\leq(F^\times)^p, \quad O_v^\times\leq N.
\]
By Proposition \ref{existence of tame valuations} we may further assume that $\Char\,\bar F_v\neq p$ and $v$ is $p$-Henselian.
By Corollary \ref{corollary on rigidity}, $N\leq T$, so
\[
(\Gam_v:p\Gam_v)=(F^\times:(F^\times)^pO_v^\times)\geq(F^\times:N)\geq(F^\times:T).
\]
From (\ref{formula for m}) we deduce that $\dim_{\dbF_p}(\Gam_v/p\Gam_v)\geq m$.

\medskip

(b) \quad
Assume first that $-1\in(F^\times)^p$.
Then $N=(F^\times)^p$.
Theorem \ref{Efr0626 5 5}(b) yields a valuation $v$ on $F$ such that $1+\grm_v\leq(F^\times)^p$ and $((F^\times)^pO_v^\times:(F^\times)^p)|p$.
As in (a), we may assume that $\Char\,\bar F_v\neq p$ and $v$ is $p$-Henselian.
The kernel of $F^\times/(F^\times)^p\to\Gam_v/p\Gam_v$ is $(F^\times)^pO_v^\times/(F^\times)^p$, which has the asserted size.

Next assume that $-1\not\in (F^\times)^p$.
Then $p=2$ and $N=\langle -1,(F^\times)^2\rangle$.
As in the previous case, there is a $2$-Henselian valuation $v$ on $F$ such that $N$ is a subgroup of $(F^\times)^2O_v^\times$ of index dividing $2$, and $\Char\,\bar F_v\neq 2$.
The above kernel is $(F^\times)^2O_v^\times/(F^\times)^2$, and we have
\[
((F^\times)^2O_v^\times:(F^\times)^2)|2(N:(F^\times)^2)=4. 
\]
(c) \quad
We argue as in (b), using Lemma \ref{AEJ lemma}.
\end{proof}

\section{Lego ETC + Demu\v skin $\Rightarrow$  Arithmetical ETC}
\label{section on equivalence of conjectures}
We prove the opposite implication to Corollary \ref{cor to ETC} and Theorem \ref{ETC implies Demuskin conjecture}:

\begin{thm}
The `Lego form' of the Elementary type Conjecture (Conjecture \ref{ETC Lego version}) together with the Demu\v skin Conjecture (Conjecture \ref{strong conjecture on Demuskin fields}) imply the arithmetical form of the Elementary Type Conjecture (Conjecture \ref{ETC arithmetical form}).
\end{thm}
\begin{proof}
Let $F$ be a field containing a root of unity of order $p$ and such that $G_F(p)$ is finitely generated.
By Conjecture \ref{ETC Lego version}, $G_F(p)$ can be written
as 
\[
\calG_F(p)=\calG_1*\cdots*\calG_n,
\]
where for each $1\leq i\leq n$, the cyclotomic pro-$p$ pair $\calG_i$ is one of the building blocks (1)--(4) of Definition \ref{Definition of ETp}, or has the form $\dbZ_p\rtimes\bar\calG$ for some $\bar\calG=\bar\calG_i$.
We may assume that none of the $\calG_i$ is the trivial pair $(1,1)$.
When $\calG_i=\dbZ_p\rtimes\bar\calG_i$ we may also assume that $\bar\calG_i$ is nontrivial, since otherwise $\calG_i=\calZ^1$.
Also, when $p=2$, we may assume that $\bar\calG_i\neq\calE$, since otherwise, we can further decompose $\calG_i\isom\calE*\calE$ (Example \ref{Z2 E is E E}).

Write $\calG_i=\calG_{F_i}(p)$ for a subfield $F\subseteq F_i\subseteq F(p)$.
We show that in all cases, the underlying group $G_i$ of $\calG_i$ is as in Conjecture \ref{ETC arithmetical form}.

\medskip

{\sl Case I: \quad $\calG_i\isom\calZ^\alp$ for some $\alp\in\dbZ^{\times,1}$, or $\calG_i\isom\calE$, $p=2$.} \quad
Then the claim is immediate.

\medskip

{\sl Case II: \quad $\calG_i$ has $p$-adic type.} \quad
Then, by Conjecture \ref{strong conjecture on Demuskin fields}, $\calG_i\isom\calG_{F_i}(p)$ for some arithmetically Demu\v skin field $F_i$.
In this case we take $u,w$ as in Definition \ref{def of arithmetically Demuskin fields}, and let $v$ be again the refinement of $u$ with $w=v/u$.
Recall that $u$, $w$, and $v$ are $p$-Henselian (Proposition \ref{properties of the completion}(c)).
For the completion $(\hat E,\hat w)$  of $(E,w)$ we have $G_E(p)\isom G_{\hat E}(p)$ via restriction, by Proposition \ref{properties of the completion}(b).
Since the residue fields of $w,\hat w$ are the same, the restriction maps the inertia group of $\hat w$ in $\hat E(p)$ isomorphically onto the inertia subgroup of $w$ in $E(p)$ (see (\ref{exact sequence for decomposition group})).
Since the inertia group of $\hat w$ is nontrivial, so is the inertia group of $w$.
As $\Gam_u=p\Gam_u$ and $\Char\,\bar F_u=0$, the inertia group of $u$ is trivial (Proposition \ref{extensions from tame valuations}).
By Lemma \ref{inertia groups}, the inertia group of $v$ is isomorphic to that of $w$, whence is also nontrivial, as required.

\medskip

{\sl Case III:  \quad $\calG_i\isom\dbZ_p\rtimes\bar\calG$ for some pair $\bar\calG$, and $(F_i^\times)^p$ is not totally rigid in $F_i$.} \quad
Here Theorem \ref{existence of valuations}(a) yields a $p$-Henselian valuation $v$ on $F_i$ such that $\Char\,(\overline{F_i})_v\neq p$ and $\Gam_v\neq p\Gam_v$.
Then $G_{F_i}(p)$ is the decomposition group of $v$ with respect to its unique extension to $F(p)=F_i(p)$ and the inertia group is $\dbZ_p^{\dim(\Gam_v/p\Gam_v)}\neq1$ (Proposition \ref{extensions from tame valuations}).

\medskip

{\sl Case IV:  \quad $\calG_i\isom\dbZ_p\rtimes\bar\calG$ for some nontrivial pair $\bar\calG$, $(F_i^\times)^p$ is totally rigid in $F_i$, and $-1\in (F_i^\times)^p$.} \quad
Then $p^2|(F_i^\times:(F_i^\times)^p)$.
Theorem \ref{existence of valuations}(b) yields a $p$-Henselian valuation $v$ on $F_i$ such that $\Char\,(\overline{F_i})_v\neq p$ and
the natural map $F^\times/(F^\times)^p\to\Gam_v/p\Gam_v$ has kernel of size dividing $p$.
Therefore $\Gam_v\neq p\Gam_v$, and we proceed as in Case III.

\medskip

{\sl Case V:  \quad $\calG_i\isom\dbZ_p\rtimes\bar\calG$ for some nontrivial pair $\bar\calG=(\bar G,\bar\theta)\not\isom\calE$, the subgroup $(F_i^\times)^p$ is totally rigid in $F_i$, and $-1\not\in (F_i^\times)^p$.} \quad
Then $p=2$.
Theorem \ref{existence of valuations}(b) yields a $2$-Henselian valuation $v$ on $F_i$ such that $\Char\,(\overline{F_i})_v\neq 2$ and
the map $F^\times/(F^\times)^2\to\Gam_v/2\Gam_v$ has kernel of size dividing $4$.
As before, this proves the assertion for $\dim_{\dbF_2}(F_i^\times/(F_i^\times)^2)\geq3$, i.e., $\dim_{\dbF_2}H^1(\bar G)\geq2$.

Since $\bar\calG$ is nontrivial, $\dim_{\dbF_2}H^1(\bar G)\neq0$.

It remains to consider the case $\dim_{\dbF_2}H^1(\bar G)=1$.
As $\bar\calG\not\isom\calE$, this implies that $\bar\calG\isom\calZ^\alp$ for some $\alp\in\dbZ_2^\times$, so
$\calG_i=\dbZ_2\rtimes\calZ^\alp$ and $(F_i^\times\colon(F_i^\times)^2)=4$.
Since $G_{F_i}(2)$ is not generated by elements of order $2$, the field $F_i$ is not Pythagorean.
Theorem \ref{existence of valuations}(c) yields a $2$-Henselian valuation $v$ on $F_i$ with $\Char\,(\bar F_i)_v\neq2$ and $\Gam_v\neq2\Gam_v$, so we are done again by Proposition \ref{extensions from tame valuations}.
\end{proof}

\section{Applications}
\label{section on applications}
In this final section we list several conjectures in Galois cohomology and quadratic form theory that were settled in the elementary type case.

\subsection{The level conjecture}
Here we set $p=2$.
Recall that the \textsl{level} $s(F)$ of a field $F$ of characteristic $\neq 2$ is the minimal positive integer $s$ such that $-1$ is a sum of $s$ squares in $F$; if there is no such integer (equivalently, $F$ admits an ordering), one sets $s(F)=\infty$.
By a celebrated result of Pfister, if $s(F)$ is finite, then it is a power of $2$ \cite{Lam05}*{Ch.\ XI, Th.\ 2.2}.
It is an open question whether, for $F$ satisfying $(F^\times:(F^\times)^2)<\infty$ (equivalently, $G_F(2)$ is finitely generated), the level must be $1,2,4$, or $\infty$ \cite{Lam05}*{Ch.\ XIII, Question 6.2}.
This was proved by Becher for $(F^\times:(F^\times)^2)\leq 256$ \cite{Becher01}.

Now given a cyclotomic pro-$2$ pair $\calG=(G,\theta)$, we view the homomorphism $\eps_\calG\colon G\to\{\pm1\}$ of (\ref{eps calG}) as an element of $H^1(G)$.
We define the \textsl{logarithmic level} ${\rm logl}(\calG)$ of $\calG$ to be the minimal positive integer $m$ such that the $m$-fold cup product $\eps_\calG\cup\cdots\cup\eps_\calG$ is $0$ in $H^m(G)$, and to be $\infty$ if no such integer exists.
Recall that if $\calG=\calG_F(2)$ for a field $F$ of characteristic $\neq2$, then $\eps_\calG=(-1)_F$ (Remark \ref{aaa}(2)).
Further, $(-1)_F\cup\cdots\cup(-1)_F=0$ in $H^m(G_F(2))$ if and only if 
the $m$-fold Pfister form $\langle\langle 1\nek1\rangle\rangle=2^m\langle1\rangle$ is isotropic, or equivalently, $-1$ is a sum of at most $2^m-1$ squares in $F$.
By Pfister's theorem, this means that $s(F)\leq 2^{m-1}$.
Thus the above problem is whether ${\rm logl}(\calG_F(2))\in\{1,2,3,\infty\}$.
The following theorem answers this question positively when $\calG\in\ET_p$:

\begin{thm}
Every $\calG=(G,\theta)\in\ET_2$ has logarithmic level $1,2,3$ or $\infty$.
\end{thm}
\begin{proof}
We argue by induction on the structure of $\calG$.

For $\calG=(1,1)$ or $\calG=\calZ_\alp$, $\alp\in\dbZ_p^{\times,1}$, we have $H^2(G)=0$, so ${\rm logl}(\calG)\in\{1,2\}$.

For $\calG=\calE$, the cup product $\eps_\calG\cup\cdots\cup\eps_\calG\in H^m(\dbZ/2)$ is nonzero for every $m$ \cite{AdemMilgram94}*{Ch.\ II, Th.\ 4.4}, so ${\rm logl}(\calE)=\infty$.

Next take $\calG$ of $2$-adic type, so $\calG=\calG_F(2)$ where $F$ is a finite extension of $\dbQ_2$.
Then $s(F)\in\{2,4\}$ \cite{Lam05}*{Ch.\ XI, Example 2.4(7)}, so ${\rm logl}(\calG)\in\{2,3\}$.                                                                                                                                
 
Now suppose that $\calG=\calG_1*\calG_2$ and the assertion is true for $\calG_i=(G_i,\theta_i)$, $i=1,2$.
Then $H^*(G)=H^*(G_1)\coprod H^*(G_2)$ as graded rings \cite{NeukirchSchmidtWingberg}*{Th.\ 4.1.4} and $\eps_\calG$ corresponds to $(\eps_{\calG_1},\eps_{\calG_2})$ under $H^1(G)=H^1(G_1)\oplus H^1(G_2)$.
It follows that ${\rm logl}(\calG)=\max\{{\rm logl}(\calG_1),{\rm logl}(\calG_2)\}$.

Finally, suppose that $\calG=\dbZ_p\rtimes\bar\calG$, where the assertion holds for $\bar\calG$.
Then $H^1(G)=(\dbZ/2)\oplus\inf(H^1(\bar G))$  and $\eps_\calG=\inf(\eps_{\bar\calG})$ (Lemma \ref{cohomological properties of extensions}(a)).
Moreover, $\inf\colon H^*(\bar G)\to H^*(G)$ is injective (Proposition \ref{generalized Wadsworth formula}).
Hence ${\rm logl}(\calG)={\rm logl}(\bar\calG)$.
\end{proof}

\subsection{Pythagoras numbers}
Let $F$ be a field of characteristic $\neq2$.
The \textsl{Pythagoras number} $P(F)$ of $F$ is the minimal positive integer $n$ such that every sum of squares in $F$ is a sum of $n$ squares, if such $n$ exists, and is $\infty$ if there is no such $n$.
Thus $F$ is Pythagorean if and only if $P(F)=1$.
If $p=2$ and $\calG_F(2)$ has elementary type, then it can be shown that $P(F)\leq5$ \cite{Marshall04}*{\S3.4}.
It appears to be unknown if this holds under the weaker assumption that $(F^\times:(F^\times)^2)<\infty$. 
If one removes the finiteness assumption, then any positive integer can be realized as the Pythagoras number of a field \cite{Hoffmann99}.

Various other quadratic form invariants can be treated similarly. 
For instance, the \textsl{$u$-invariant} $u(F)$ of $F$, which is the maximal dimension of an anisotropic quadratic form over $F$ (and $\infty$ if there is no such maximal dimension), can be shown to be a $2$-power whenever $\calG_F(2)$ has elementary type; See \cite{Lam05}*{Ch.\ XIII, \S6}. 

\subsection{Witt rings}
We take again $p=2$.
As shown by Jacob and Ware \cite{JacobWare89}, the operations of free product and extension in the category of cyclotomic pro-$2$ pairs correspond to the operations of direct product and group ring formation in the category of abstract Witt rings \cite{Marshall80}.
Further, the building blocks in Conjecture \ref{ETC Lego version} correspond to the Witt rings of $\dbC$, finite fields, $\dbR$, and the finite extensions of $\dbQ_2$, respectively.
Therefore, as shown in \cite{JacobWare89}, Conjecture \ref{ETC Lego version} implies Marshall's Elementary Type Conjecture for finite Witt rings of fields of characteristic $\neq2$ (see \cite{Marshall80}).

\subsection{The strong Massey vanishing conjecture}
Massey products form a general ``external' construction in the homological context of differential graded algebras.
In recent years they were extensively studied  in the context of the cohomology of maximal pro-$p$ Galois groups $G=G_F(p)$, where as usual $F$ contains a root of unity of order $p$.
There, given $\varphi_1\nek\varphi_n\in H^1(G)$, the Massey product $\langle\varphi_1\nek\varphi_n\rangle$ is a subset of $H^2(G)$.
When $n=2$ we have $\langle\varphi_1,\varphi_2\rangle=\{\varphi_1\cup\varphi_2\}$.
The \textsl{Massey Vanishing Conjecture} by Min\'a\v c and T\^an predicts that if $n\geq3$, any nonempty $n$-fold Massey product in $H^2(G)$ necessarily contains $0$.
This was proved for $n=3$ by Min\'a\v c--T\^an \cite{MinacTan16} and Matzri and the author (\cite{Matzri14}, \cite{EfratMatzri17}),
and for $n=4$, $p=2$, by Merkurjev and Scavia \cite{MerkurjevScavia23}.

Now assuming further that $\sqrt{-1}\in F$ when $p=2$,  Quadrelli \cite{Quadrelli24} proves the Massey vanishing conjecture (in fact in a stronger form) for every $F$ for which $\calG_F(p)$ has elementary type.

\subsection{The Massey symbol length conjecture}
Let $F$ be again a field containing a root of unity of order $p$.
The cup product homomorphism $\cup\colon H^1(G_F(p))^{\tensor2}\to H^2(G_F(p))$ is surjective, by the Merkurjev--Suslin theorem.
Thus every element $\alp$ of $H^2(G_F(p))$ is a sum of finitely many cup products $\psi\cup\psi'$ with $\psi,\psi'\in H^1(G_F(p))$.
The minimal number of cup products required is called the \textsl{symbol length} of $\alp$.

The \textsl{Massey Symbol Length Conjecture} predicts that, given $n\geq2$, there is a uniform bound $M(n,p)<\infty$ such that for every $F$, every $\varphi_1\nek\varphi_n\in H^1(G_F(p))$, 
and every element $\alp$ of the $n$-fold Massey product $\langle\varphi_1,\nek\varphi_n\rangle$, the symbol length of $\alp$ is at most $M(n,p)$.
For instance, one trivially has $M(2,p)=1$, and by the above results of \cite{MinacTan16}, \cite{Matzri14}, and \cite{EfratMatzri17}, $M(3,p)=2$.
Merkurjev and Scavia pointed out to the author that $M(4,2)\leq4$.
In \cite{Efrat24} we prove that the conjecture holds when restricted to the class of fields $F$ such that $\calG_F(p)$ has elementary type.
In fact, the paper proves for this class a more general conjecture by Positselski about the symbol lengths of pullbacks of certain cohomology elements.

\subsection{Koszulity conjectures}
A conjecture by Positselski \cite{Positselski14} predicts that for a field $F$ containing a root of unity of order $p$, the cohomology graded algebra $H^*(G_F(p))=\bigoplus_{i\geq0}H^i(G_F(p))$ with the cup product has a strong property known as \textsl{Koszulity}.
If true, this would partly generalize the  isomorphism $K^M_*(F)/p\xrightarrow{\sim}H^*(G_K(p))$ given by the Norm Residue Theorem of Voevodsky and Rost.
In \cite{MinacPasiniQuadrelliTan21} Min\'a\v c et al prove this conjecture under the assumption that $\calG_F(p)$ has elementary type.

\end{document}